\documentclass[11pt, a4paper, reqno]{amsart}

\usepackage{amscd}
\usepackage{dsfont}
\usepackage{thmtools}
\usepackage{thm-restate}
\usepackage{amsmath, upgreek}
\usepackage{amstext}
\usepackage{amsthm}
\usepackage{amssymb}
\usepackage{amsopn}
\usepackage{amssymb}
\usepackage{amsxtra}
\usepackage{amsfonts}
\usepackage{fancyhdr}
\usepackage{paralist, epsfig}
\usepackage[mathcal]{euscript}
\usepackage[english]{babel}
\usepackage{csquotes}
\usepackage{bbold}
\usepackage[scr=rsfso]{mathalpha}
\usepackage{dutchcal}

\usepackage{setspace}
\usepackage[dvipsnames]{xcolor}
\usepackage{url}

\usepackage{tikz}
\usepackage{tikz-cd}
\usepackage{tkz-euclide}
\usepackage{euscript,mathtools}

\usepackage{hyperref}
\usepackage[capitalize]{cleveref}
\hypersetup{
    colorlinks=true,
    linkcolor=MidnightBlue,
    urlcolor=NavyBlue,
    citecolor=Maroon
}

\makeatletter
\g@addto@macro\normalsize{%
  \setlength\abovedisplayskip{10pt}
  \setlength\belowdisplayskip{10pt}
  \setlength\abovedisplayshortskip{5pt}
  \setlength\belowdisplayshortskip{8pt}
}

\newtheorem{thm}{Theorem}[section]
\Crefname{thm}{Theorem}{Theorems}
\newtheorem{cor}[thm]{Corollary}
\Crefname{cor}{Corollary}{Corollaries}
\newtheorem{lem}[thm]{Lemma}
\Crefname{lem}{Lemma}{Lemmas}
\newtheorem{prop}[thm]{Proposition}
\Crefname{prop}{Proposition}{Propositions}

\Crefname{quest}{Question}{Questions}

\Crefname{cons}{Construction}{Constructions}

\theoremstyle{definition}
\newtheorem{ex}[thm]{Example}
\Crefname{ex}{Example}{Examples}
\newtheorem{rmk}[thm]{Remark}
\Crefname{rmk}{Remark}{Remarks}
\newtheorem{dfn}[thm]{Definition}
\Crefname{dfn}{Definition}{Definitions}

\Crefname{nttn}{Notation}{Notations}

\newcommand{\NN}{\mathbb{N}}

\newcommand{\Real}{\mathbb{R}}

\newcommand{\QQ}{\mathbb{Q}}

\renewcommand{\theta}{\vartheta}
\renewcommand{\epsilon}{\varepsilon}

\newcommand{\ran}{\operatorname{ran}\nolimits}
\newcommand{\id}{\operatorname{id}\nolimits}

\DeclareMathOperator*{\colim}{colim}
\newcommand{\op}{\operatorname{op}\nolimits}

\newcommand{\Born}{\mathscr{B}}
\newcommand{\Comp}{\mathscr{C}}
\newcommand{\Domp}{\mathscr{D}}
\newcommand{\Qomp}{\mathscr{Q}}
\newcommand{\Eomp}{\mathscr{E}}

\newcommand{\Pow}{\mathscr{P}}

\DeclareFontFamily{U}{dmjhira}{}
\DeclareFontShape{U}{dmjhira}{m}{n}{
  <-> dmjhira
}{}
\DeclareFontSubstitution{U}{dmjhira}{m}{n}
\DeclareRobustCommand{\yo}{\textup{\usefont{U}{dmjhira}{m}{n}\symbol{"48}}}
\newcommand\cmpuly[1]{(#1)_{\operatorname{cmp}}}
\newcommand\topuly[1]{(#1)_{\operatorname{top}}}

\newcommand{\catstyle}[1]{\mathbcal{\hspace{0.05em}{#1}\hspace{0.05em}}}

\newcommand{\catC}{\catstyle{C}}
\newcommand{\catD}{\catstyle{D}}
\newcommand{\catI}{\catstyle{I}}
\newcommand{\catJ}{\catstyle{J}}

\newcommand{\Fun}{\catstyle{F\hspace{-0.1em}un}}
\newcommand{\Set}{\catstyle{Set}}
\newcommand{\Top}{\catstyle{T\!op}}
\newcommand{\Grp}{\catstyle{Grp}}

\newcommand{\RMod}{\catstyle{Mod}_{\!R}}

\newcommand{\KAlg}{\catstyle{Alg}_{\mathbb{k}}}

\newcommand{\PSh}[1]{{\Fun(#1^{\op}, \Set)}}
\newcommand{\Cond}{{\catstyle{Cond}}}
\newcommand{\qsCond}{{\catstyle{qsCond}}}
\newcommand{\Cmp}{\catstyle{Cmp}}

\newcommand{\CHaus}{\catstyle{CHaus}}
\newcommand{\cgTop}{\catstyle{CG}}
\newcommand{\CGWH}{\catstyle{C\hspace{-0.1em}G\hspace{-0.1em}W\hspace{-0.1em}H}}

\newcommand{\into}{\hookrightarrow}

\newcommand{\onto}{\twoheadrightarrow}

\newcommand{\adjarrow}{\longleftrightarrow}

\usepackage{xstring}
\DeclareMathOperator*{\fprod}{%
  \mathchoice% http://tex.stackexchange.com/q/148740/5764
    {\raisebox{-.0\height}{\scalebox{1.15}{$\prod$}}}% \displaystyle
    {\raisebox{-.0\height}{\scalebox{1}{$\prod$}}}% \textstyle
    {\raisebox{-.0\height}{\scalebox{0.75}{$\prod$}}}% \scriptstyle
    {\raisebox{-.05\height}{\scalebox{0.55}{$\prod$}}}}% \scriptscriptstyle

\DeclareMathOperator*{\fsqcup}{%
  \mathchoice% http://tex.stackexchange.com/q/148740/5764
    {\raisebox{-.0\height}{\scalebox{1.15}{$\bigsqcup$}}}% \displaystyle
    {\raisebox{-.0\height}{\scalebox{1}{$\bigsqcup$}}}% \textstyle
    {\raisebox{-.0\height}{\scalebox{0.75}{$\bigsqcup$}}}% \scriptstyle
    {\raisebox{-.05\height}{\scalebox{0.55}{$\bigsqcup$}}}}% \scriptscriptstyle

\DeclareMathOperator*{\fcoprod}{%
  \mathchoice% http://tex.stackexchange.com/q/148740/5764
    {\raisebox{-.0\height}{\scalebox{1.15}{$\coprod$}}}% \displaystyle
    {\raisebox{-.0\height}{\scalebox{1}{$\coprod$}}}% \textstyle
    {\raisebox{-.0\height}{\scalebox{0.75}{$\coprod$}}}% \scriptstyle
    {\raisebox{-.05\height}{\scalebox{0.55}{$\coprod$}}}}% \scriptscriptstyle

\renewcommand*\textcircled[1]{\tikz[baseline=(char.base)]{
            \node[shape=circle,draw,inner sep=1pt] (char) {#1};}}

	\usetikzlibrary{matrix,arrows,decorations,decorations.pathmorphing,positioning,decorations.pathreplacing,shapes}
	\tikzset{commutative diagrams/.cd, 
		mysymbol/.style = {start anchor=center, end anchor = center, draw = none}}

\DeclareMathSymbol{\sm}{\mathbin}{AMSa}{"39}
\DeclareMathSymbol{\shortminus}{\mathbin}{AMSa}{"39}

\newcommand{\diam}{\hfill\raisebox{-0.75pt}{\scalebox{1.4}{$\diamond$}}}

\usepackage{enumitem}

\newlist{myitemize}{itemize}{1}
\setlist[myitemize,1]{leftmargin=26pt, noitemsep, topsep=0pt}

\DeclareMathOperator*{\fcup}{%
  \mathchoice% http://tex.stackexchange.com/q/148740/5764
    {\raisebox{-.0\height}{\scalebox{1.15}{$\bigcup$}}}% \displaystyle
    {\raisebox{-.0\height}{\scalebox{1}{$\bigcup$}}}% \textstyle
    {\raisebox{-.0\height}{\scalebox{0.75}{$\bigcup$}}}% \scriptstyle
    {\raisebox{-.05\height}{\scalebox{0.55}{$\bigcup$}}}}% 

\DeclareMathOperator*{\fcap}{%
  \mathchoice% http://tex.stackexchange.com/q/148740/5764
    {\raisebox{-.0\height}{\scalebox{1.15}{$\bigcap$}}}% \displaystyle
    {\raisebox{-.0\height}{\scalebox{1}{$\bigcap$}}}% \textstyle
    {\raisebox{-.0\height}{\scalebox{0.75}{$\bigcap$}}}% \scriptstyle
    {\raisebox{-.05\height}{\scalebox{0.55}{$\bigcap$}}}}% \scriptscriptstyle

\begin{document}
\allowdisplaybreaks
$ $
\vspace{-40pt}

\title{Condensed mathematics through\\[2pt]compactological spaces}

%\author{Author 1\hspace{0.5pt}\MakeLowercase{$^{\text{1}}$} and Author 2\hspace{0.5pt}\MakeLowercase{$^{\text{2}}$}}

\author{Franziska B\"ohnlein\hspace{0.5pt}$^{\text{1,\hspace{0.5pt}2}}$, Benjamin Bruske\hspace{0.5pt}$^{\text{1,\hspace{0.5pt}2}}$ and Sven-Ake\ Wegner\hspace{0.5pt}$^{\text{1,\hspace{0.5pt}2,\hspace{0.5pt}3}}$}

%\hspace{0.5pt}\MakeLowercase{$^{\text{1}}$}

\renewcommand{\thefootnote}{}
\hspace{-1000pt}\footnote{\hspace{5.5pt}2020 \emph{Mathematics Subject Classification}: 46M15, 46A17, 18F60, 18A25, 54B30, 54E99, 55U40.\vspace{1.6pt}}

% 46M15 Categories, functors in functional analysis
% 46A17 Bornologies and related structures; Mackey convergence, etc.

% 18F60 Categories of topological spaces and continuous mappings [See also 54-XX]
% 18A25 Functor categories, comma categories

% 54B30 Categorical methods in general topology
% 54E99 Topological spaces with richer structures: None of the above, but in this section

% 55U40 Topological categories, foundations of homotopy theory

\hspace{-1000pt}\footnote{\hspace{5.5pt}\emph{Key words and phrases}: compactological space, condensed set, bornological space, exact competion. \vspace{1.6pt}}

\hspace{-1000pt}\footnote{\hspace{0pt}$^{1}$\,University of Hamburg, Department of Mathematics, Bundesstra\ss{}e 55, 20146 Hamburg, Germany; franzis-\newline\phantom{x}\hspace{1.2pt}ka.boehnlein@uni-hamburg.de, benjamin.bruske@studium.uni-hamburg.de, sven.wegner@uni-hamburg.de.\vspace{1.6pt}}

\hspace{-1000pt}\footnote{\hspace{0pt}$^{2}$\,This research was funded by the German Research Foundation\,(DFG); Project no.~507660524.\vspace{1.6pt}}

\hspace{-1000pt}\footnote{\hspace{0pt}$^{3}$\,Corresponding author.}

\begin{abstract}
In their 2022 lecture notes on condensed sets, Clausen and Scholze mentioned in a remark that the important subclass of quasiseparated condensed sets is equivalent to the category of so-called compactological spaces defined by Waelbroeck in the 1960s. In this paper we study the latter category in detail, give a rigorous proof of Clausen and Scholze's claim, and establish that condensed sets are a formal categorical completion of Waelbroeck's compactological spaces. The latter answers a question asked by Hanson in 2023 and permits the interpretation of compactological sets as an `elementary' approach to condensed mathematics.
\end{abstract}

\vspace{-15pt}

\maketitle

\vspace{-30pt}

\section{Introduction}\label{SEC-Intro}

Compactological sets, back then under the name of compactological spaces, were introduced in the context of functional analysis by Waelbroeck \cite{Waelbroeck60, Waelbroeck1967, Waelbroeck} and Buchwalter \cite{Buchwalter65, BuchwalterDiss, Buchwalter69} in the 1960s as a set equipped with a collection of subsets, that all carry the structure of compact Hausdorff spaces, and which satisfy natural coherence conditions. Formally, we write $(X,\Comp,(\uptau_K)_{K\in\Comp})$ for a compactological set and refer to the collection $\Comp$ as its compactology. Due to the coherence conditions, the $\uptau_K$ induce a topology $\uptau$ on $X$ such that all $K\in\Comp$ are compact subspaces of $(X,\uptau)$; the collection $\Comp$ thus generalizes the compact subsets of a topological (cgwh-)space. Indeed, compactological sets refine the latter class of spaces, since designating the collection $\Comp$ imprints an additional structure onto the topological space. From another perspective, being contained in a set that belongs to $\Comp$ may be regarded as a notion of `smallness' that, again due to the coherence conditions, interacts well with the topological notion of closedness. The small sets in this sense form a bornology as introduced and studied extensively by Waelbroeck, Buchwalter (loc.~cit.), Hogbe-Nlend \cite{Hogbe70,Hogbe-French, Hogbe} and Houzel \cite{Houzel}.

\smallskip

The aforementioned authors not only considered sets endowed with compactologies or bornologies, but extended the notion to vector spaces and other algebraic structures. In particular bornological methods play an active role at the interface of functional analysis, homological algebra and geometry; see, e.g.~\cite{Meyer2, BOBB, BBBK18, B19, KKM, ben2024perspective, Stempfhuber, KNW25, Savage} for a sample of recent publications. In a nutshell, the reason for this is that, from a categorical point of view, bornological vector spaces behave significantly better than their topological siblings.

\smallskip

Another approach to bypass the somewhat obstinate categorical properties of (locally convex) topological vector spaces, groups or modules is the paradigm of condensed mathematics due to Clausen and Scholze \cite{ScholzeCondensedMath}, which is\,---\,surprisingly or not\,---\,intimately related to Waelbroeck and Buchwalter's theory of compactological sets. In this article we will develop the theory of compactological sets from a topological perspective, construct and prove the categorical equivalence between the category of compactological sets and the category of quasiseparated condensed sets, as noted without proof in \cite[p.~15]{ScholzeComplexGeometry}, and finally extend this result to a universal presentation of condensed sets as a formal categorical completion of compactological sets. The latter answers a question asked by Hanson \cite{Hanson}.

\smallskip

In this article we will focus on the relation between condensed and compactological sets; the case of condensed and compactological groups, rings, modules, etc.\ will be treated in a forthcoming paper. Based on the `topological foundation' provided here, the second paper will offer an accessible characterization of the condensed algebraic objects of interest, extending the main theorem of the present article and connecting it with previous work of Wegner \cite{Wegner17} in order to obtain a simple and workable description in terms of fractions. We will show that these can be manipulated efficiently and thereby provide a means to perform computations in the condensed world in terms of classical objects. This is of particular interest as the universal connection implies that the derived categories of compactological modules and of condensed modules over a fixed quasiseparated ring agree, which may simplify many derived calculations.

\smallskip

Let us now give a short summary of the article and its main results. In Section \ref{SEC:Cmp} we will start by spelling out Buchwalter's definition of compactological sets (this is the notion we will use throughout the article) and give an overview of different, but equivalent, definitions. We include a new characterization which shows that the whole data of a compactological set $(X,\Comp,(\uptau_{K})_{K\in\Comp})$ is already encoded in the system $\Comp$ of `plain sets', see \cref{DFN-COMP-4}. We then give explicit descriptions of subspaces, quotients, products, coproducts, limits and colimits within the category $\Cmp$ of compactological sets and show in particular that the latter is complete and cocomplete, see \cref{cor:cmp-bicomp}. In parallel we investigate the relation to the category $\CGWH$ of compactly generated weak Hausdorff spaces in the sense defined by McCord \cite{McCord}. This amounts to the following result in which $\upiota$ endows a cgwh-space with the collection $\Comp$ of all compacts while $\topuly-$ takes a compactological set $(X,\Comp,(\uptau_{K})_{K\in\Comp})$, furnishes $X$ with the final topology of the inclusions $(K,\uptau_K)\hookrightarrow X$, and then forgets $\Comp$ and the $\uptau_K$.

\begin{restatable*}[]{thm}{ThmAdjCGWHCmp}\label{CGWH-COMPL-ADJ} 
    The two functors $\topuly-$ and $\upiota$ are well-defined and form an adjoint pair:
    \begin{center}
        \begin{tikzcd}[column sep=large, ampersand replacement=\&]
           \CGWH
            \arrow[r, "\upiota"{name=F}, bend left=17] \&
            \Cmp.
            \arrow[l, "(-)_{\operatorname{top}}"{name=G}, bend left=17]
            %--- Adjunction Symbol
            \arrow[phantom, from=F, to=G, "\dashv" rotate=90]
        \end{tikzcd}
    \end{center}
    Moreover, $\upiota$ is fully faithful, it holds $(\upiota(-))_{\operatorname{top}} = \id_{\CGWH}$ and therefore we may read $\CGWH$ as a reflective subcategory of $\Cmp$. In particular, we may canonically understand the compact spaces $K\in\CHaus\into\CGWH\into\Cmp$ as compactological sets endowed with the compactology of all compact subsets, each carrying the  topology induced by $K$.
\end{restatable*}

\noindent We conclude Section \ref{SEC:Cmp} by showing that every compactological space can be canonically written as a filtered colimit of compact Hausdorff spaces with embeddings as transition maps, or, more precisely, that $\Cmp\cong\catstyle{Ind_{emb}}(\CHaus)$ holds, see Propositions \ref{prop:cmp-gen-compact} and \ref{prop:ind-cmp}. 

\smallskip

Section \ref{SEC:Equiv} starts with a recap of the basics on condensed sets. We follow Clausen and Scholze's first set of lecture notes \cite{ScholzeCondensedMath} and repeat several fundamental results without proofs for later reference. As the study of functor categories requires caution in order to avoid set-theoretic conflicts, we have to work with some suitable size restriction, i.e., consider for example only the category $\CHaus_{\upkappa}$ of compact Hausdorff spaces of cardinality less than $\upkappa$. The main result of Section \ref{SEC:Equiv} is the following theorem, in which we prove the aforementioned equivalence between the category of compactological condensed sets and the category of quasiseparated condensed sets. Below, $\upkappa$ is assumed to be a strong limit cardinal, $\yo(C):=\Cmp_{\upkappa}(-,C)|_{\CHaus}$ is the Yoneda embedding restricted to compact Hausdorff spaces, and $\cmpuly{X}$ is the underlying set endowed with the compactology $\Comp_{X} := \{ \ran(\upphi_*)\subseteq X(*) \mid \upphi \in X(K),\, K\in\CHaus_\upkappa\}$.

\begin{restatable*}[]{thm}{thmcmpcondqs}\label{thm:cmp-condqs}  
    The restricted Yoneda embedding and assigning the underlying compactological set to a condensed set are quasi-inverses inducing an equivalence between $\upkappa$-compactological and quasiseparated $\upkappa$-condensed sets:
    \[
        \yo\colon \Cmp_\upkappa \overset\sim\adjarrow \qsCond_\upkappa\hspace{1pt}:\cmpuly-.
    \]
    These equivalences extend along the directed colimits to an equivalence between the unrestricted category $\Cmp$ of all compactological sets and the category $\qsCond$.%, as constructed in \cite[Def.~2.11]{ScholzeCondensedMath}}
\end{restatable*}

\noindent{} We refer to \cref{rem:cmp-size,rem:cond-cond} for details on the categories $\qsCond$ and $\Cond$ of `all condensed sets', i.e. without any size restriction. Let us mention that also Stempfhuber \cite{Stempfhuber} recently surveyed the correspondence between quasiseparated condensed and compactological sets. His work builds on Waelbroeck's definition while we follow Buchwalter's approach, cf.~\cref{RMK-HIST}; the latter aligns better with the notion of cgwh-spaces and thus gives a more topological perspective. Furthermore, we provide some details that we believe to be both genuinely new and of independent mathematical interest.

\smallskip

After examining compactological sets in relation to condensed objects, Section \ref{SEC:CmpCat} will analyze the category of compactological sets in its own right. We demonstrate that it inherits many good properties from $\CGWH$, being, among other things, total, cartesian closed, regular, coregular and extensive. Finally, we universally extend $\Cmp$ to a Barr-exact category. The results of \cref{SEC:Cmp,SEC:Equiv}, taken together, culminate in the third main theorem of this paper, which resolves a question of James E.\ Hanson \cite{Hanson} and precisely delineates how condensed sets are a natural extension of topological objects:

%After examining compactological sets in relation to condensed objects, Section \ref{SEC:CmpCat} will analyze the category of compactological sets in its own right. We demonstrate that it inherits many nice properties from the category of cgwh-spaces, being among others total, cartesian closed, regular, coregular and extensive. Finally, we investigate the result of forcing the category to be any \enquote{nicer} by universally extending it to a Barr-exact category. We obtain the main result of this paper, tying together \cref{SEC:Cmp,SEC:Equiv}, \NOTE{Benjamin}{I'm only semi-happy with this. Suggestions?}\textcolor{red}{offering a further topologically minded justification for the use of condensed sets,} and affirming James E. Hanson's hypothesis on Mathematics Stack Exchange:

\begin{restatable*}[]{thm}{thmcmpregex}\label{thm:cmp-reg-ex}
    The category $\Cond_\upkappa$ of $\upkappa$-condensed sets is the \emph{ex/reg} completion of the category $\Cmp_\upkappa$ of $\upkappa$-compactological sets. Explicitly, it is the universal Barr-exact category containing compactological sets as a full subcategory.
    The functor $\yo\colon\Cmp_\upkappa\into\Cond_\upkappa$, embedding the category of compactological sets into its \emph{ex/reg} completion, preserves coequalizers of kernel pairs and is hence a regular functor. 
    Analogously, the category $\Cond$ of all condensed sets is the \emph{ex/reg} completion of the category $\Cmp$ of all compactological sets.

    %\textcolor{red}{This result canonically extends along the directed colimits to an equivalence between the ex/reg completion of the category $\Cmp$ of all compactological sets and the category $\Cond$ of all condensed sets, as constructed in \cite[Def.~2.11]{ScholzeCondensedMath}.}
\end{restatable*}

%\vspace{10mm}

%\begin{enumerate}
%    \item Category Theory: \cite{CatContext,MacLane}.
%    \item Topology: General: \cite{Dieck}, \cite{Kelley}, cgwh-spaces: \cite[Appendix A]{Schwede}, \cite{CGWH}, \cite{McCord}, \cite{Rezk}, \cite{Lewis}.
%    \item Cardinality Topics: Shulman \cite{Shulman}
%    \item Condensed Stuff: Asgeirsson \cite{Asgeir}, \cite{ScholzeCondensedMath, ScholzeComplexGeometry}, \cite[Chapter 6]{Stempfhuber}, \cite{Commelin}
%    \item Barwick and Haine study the same construction in order to study topological objects in the setting of $\infty$-category \cite{Barwick}
%    \item Bornology: \cite{Hogbe, Hogbe-French}
%    \item Hanson's question \cite{Hanson}
%    \item Abgrenzung zu Stempfhuber: Topologische Perspektive (Benutzung von Buchwalter's Definition waehrend Stemppfhuber Waelbroeck's Definition benutzt) und wir geben detaillierte Beweise die ohne Referenz auf Hogbe-Nlend auskommen (stattdessen Argumente die die cgwh-Eigenschaft nutzen). Minor corrections. We demonstrate that compactological spaces behave very similarly to cgwh and thus might play a similar role as cgwh in algebraic topology. We give a proof from first principles for 3.17.
%\end{enumerate}

From the functional-analytic viewpoint, Theorem \ref{thm:cmp-condqs} identifies the classical compactological framework with the quasiseparated part of condensed mathematics while Theorem \ref{thm:cmp-reg-ex} describes the passage to arbitrary condensed sets by exact completion. They provide the set-theoretic foundation for the comparison of compactological and condensed algebraic objects; details have to be left to subsequent work.

\smallskip

For unexplained notation from category theory we refer to \cite{CatContext, MacLane}. Concerning topology we follow the standard notation, e.g.~\cite{Dieck, Kelley}, and refer to \cite{Schwede,CGWH,McCord,Rezk,Lewis} for details on cgwh-spaces. For readers unfamiliar with cardinality issues and how to resolve them we recommend \cite{Shulman}. Besides the original lecture notes by Clausen and Scholze, more detailed introductions to condensed mathematics can be found in \cite{Asgeir, Stempfhuber, LeStum}. 

\smallskip

Disclaimer: In order to make this article accessible to an audience as wide as possible, we include several well-known results at the beginning and give complete and self-contained proofs from first principles in Sections \ref{SEC:Cmp}\hspace{1pt}--\hspace{1pt}\ref{SEC:Equiv}. In particular, we avoid references to the original sources by Waelbroeck and Buchwalter, as they are in places arduous to read due to varying nomenclature, notation and language. In Section \ref{SEC:Cmp} we avoid heavy categorical machinery to promote condensed sets also among less categorically inclined mathematicians\,---\,using Theorems \ref{thm:cmp-condqs} and \ref{thm:cmp-reg-ex} as black boxes, Section \ref{SEC:Cmp} can be read as an `elementary' approach to condensed mathematics. 

\smallskip

Finally, we would like to point out that there is a parallel theory to condensed mathematics due to Barwick and Haine \cite{Barwick}. Their work agrees in much of the setup, but takes a different perspective in its presentation. Beyond this, characteristic of their work and contrasting the theory of condensed mathematics is the use of strongly inaccessible cardinals, of which sufficiently many are assumed to exist, as well as the choice to work in the setting of $\infty$-categories from the beginning. In the present work we purposefully restrict ourselves to a $1$-categorical perspective in order to keep the entry barrier low.  We have further opted for the more explicit treatment of set-theoretic conflicts in line with Scholze and Clausen as these specialize to the approach based on multiple universes, but, in their full generality, require some additional care. Because of this alignment we stick to the terminology used by Scholze and Clausen throughout the article.

% \section{Preliminaries}\label{SEC:Pre}
%
% General reference to category theory(\cite{CatContext,MacLane}) and topology books; some specific references for topics beyond the usual cat theory (e.g. Ind-objects, localization, $\kappa$-bimbam). Fix notion of compactness, recap CGWH, refer to Strickland (\cite{CGWH}). 

\section{Introduction to compactological sets}\label{SEC:Cmp}

Throughout the whole article we will follow the `French convention', i.e., a quasi-compact topological space satisfies the covering property and a compact topological space is additionally Hausdorff. Below we start with Buchwalter's 1969 definition of what he called a `compactologie', respectively an `espace compactologique'. We will however use the term `compactological \emph{set}' since it is more in line with the language of condensed mathematics.

\begin{dfn}[Buchwalter's definition {\cite[Dfn 1.1.1]{Buchwalter69}}]\label{DFN-COMP-0} Let $X$ be a set. A \textit{compactology} $\Comp$ on $X$ is a non-empty collection of subsets $\Comp\subseteq\mathcal{P}(X)$ endowed with topologies $\uptau_K$ such that each $(K,\uptau_K)$ is compact (and thus in particular Hausdorff) and the following four properties hold:
    \begin{enumerate}
        \item $\Comp$ is a covering of $X$: $\fcup_{K\in\Comp}K=X$,\vspace{2pt}
        \item $\Comp$ is upwards filtering: $\forall\:K_1,K_2\in\Comp\;\exists\:K_3\in\Comp\colon K_1\subseteq K_3\text{ and }K_2\subseteq K_3$, \vspace{2pt}
        \item The topologies are coherent: $\forall\:K_1,K_2\in\Comp \text{ with } K_2\subseteq K_1\colon\uptau_{K_2}=\uptau_{K_1}|_{K_2}$,\vspace{2pt}
        \item $\Comp$ is stable under compact subsets: $\forall\:K_1\in \Comp,\,K_2\subseteq(K_1,\uptau_{K_1})$ compact$\colon K_2 \in  \Comp$.\vspace{1pt}
    \end{enumerate}
The triple $C = (X, \Comp, (\uptau_K)_{K \in \Comp})$ is said to be a \textit{compactological set}.\hfill\diam{}
\end{dfn}

For later reference let us put the following easy observations into writing.

\begin{lem}\label{FIRST-LEM} Let $(X,\Comp,(\uptau_{K})_{K\in\Comp})$ be a compactological set. Then $\Comp$ is stable under finite unions and finite intersections.
\end{lem}
\begin{proof} For $K_1,K_2\in\Comp$ we pick $K_3$ as in \ref{DFN-COMP-0}(2) and observe $\uptau_{K_3}|_{K_i}=\uptau_{K_i}$ for $i=1,2$ by \ref{DFN-COMP-0}(3). Thus, $K_1,K_2\subseteq(K_3,\uptau_{K_3})$ are compact. Applying \ref{DFN-COMP-0}(4) we get $K_1\cup K_2,K_1\cap K_2\in\Comp$.
\end{proof}

We add the following partly historical comments.

\begin{rmk}\label{RMK-HIST}\begin{myitemize}\item[(1)] A frequently cited reference for compactologies is Waelbroeck's textbook \cite[Chapter III]{Waelbroeck}. His definition requires $\Comp$ to be a bornology (i.e., covering, stable under finite unions, stable under arbitrary subsets), that every $K\in\Comp$ carries a topology $\uptau_K$ satisfying the coherence condition \ref{DFN-COMP-0}(3) and that every $L\in\Comp$ is contained in a $K\in\Comp$ such that $(K,\uptau_{K})$ is compact.\vspace{3pt}

\item[(2)] The two notions are not strictly speaking equivalent, but a compactology in the Waelbroeck sense is completely determined by its subsystem of compacts, which is a compactology in the sense of Buchwalter. On the other hand, we may pass from a compactology $\Comp$ in the sense of Buchwalter to a compactology in the Waelbroeck sense by forming the set of all subsets of members of $\Comp$.\vspace{3pt}

\item[(3)] It seems that Waelbroeck \cite[Dfn 4.1]{Waelbroeck1967} was the first to study compactologies, but back then under the name of a `compact boundedness'.\vspace{3pt}

\item[(4)] In the remainder we will hold on to \cref{DFN-COMP-0} but also refer to results given by Waelbroeck and others in the slightly different setting outlined above.\hfill\diam{}
\end{myitemize}
\end{rmk}

The first idea for an example of a compactological set is to start with a topological space and to take for $\Comp$ its compact subsets endowed with the induced topologies. For this to form a well-defined compactology, a certain separation property is necessary, cf.~\cref{rem:cgwh-cmop}(2)--(3). Let us first show that every compactology does indeed induce a topology on the whole set which is coherent with the topologies on its elements. The following result was mentioned in \cite[Prop III.1]{Waelbroeck}; we include a proof for the sake of completeness.

\begin{prop}\label{cmp-final} Let $(X, \Comp, (\uptau_K)_{K \in \Comp})$ be a compactological set. Let $\uptau$ be the final topology on $X$ with respect to the inclusions $\{(K,\uptau_K) \hookrightarrow X \ | \ K \in \Comp \}$. Then $\uptau_K=\uptau|_K$ holds for every $K\in\Comp$.
\end{prop}

\begin{proof} By definition, a set $B\subseteq X$ is $\uptau$-closed iff $B\cap K'$ is $\uptau_{K'}$-closed for any $K'\in\Comp$.

\smallskip

\textcircled{1} As a preparation we first show that every $K\in\Comp$ is $\uptau$-closed: Fix $K\in\Comp$. For arbitrary $K'\in\Comp$ then $K\cap K'\in\Comp$ holds by \cref{FIRST-LEM} and it follows $\uptau_{K\cap K'}=\uptau_{K'}|_{K\cap K'}$ from \ref{DFN-COMP-0}(3). We thus get that $K\cap K'\subseteq(K',\uptau_{K'})$ is closed.

\smallskip

\textcircled{2} Now we show that for fixed $K\in\Comp$ a subset $B\subseteq K$ is $\uptau_K$-closed iff it is $\uptau|_K$-closed:

\smallskip

\textquotedblleft{}$\Longrightarrow$\textquotedblright{} If $B\subseteq(K,\uptau_K)$ is closed, then we have $B\in\Comp$ by \ref{DFN-COMP-0}(4). Let $K'\in\Comp$ be arbitrary. By \ref{DFN-COMP-0}(2) there exists $L\in\Comp$ with $K,K'\subseteq L$. Using \ref{DFN-COMP-0}(3) we get that $B$ and $K'$, and thus also $B\cap K'$, are closed in $(L,\uptau_L)$. Using $\uptau_{K'}=\uptau_{L}|_{K'}$ the claim follows.

\smallskip

\textquotedblleft{}$\Longleftarrow$\textquotedblright{} If $B\subseteq(K,\uptau|_{K})$ is closed, then \textcircled{1} implies that $B\subseteq(X,\uptau)$ is closed. By definition, we get that $B\cap K\subseteq(K,\uptau_K)$ is closed.
\end{proof}

The above suggests the following, slightly different, approach to compactological sets.

\begin{rmk}\label{RMK-COMMELIN}\begin{myitemize} \item[(1)] An equivalent definition of compactological sets arises through replacing in \cref{DFN-COMP-0} the system $(\uptau_K)_{K\in\Comp}$ of topologies by a single topology $\uptau$ on $X$ and requiring, instead of \ref{DFN-COMP-0}(3), that $\uptau$ is final with respect to the inclusions $\{(K,\uptau|_{K})\hookrightarrow X\ | \ K\in\Comp\}$. \vspace{3pt}

\item[(2)] For the Waelbroeck setting, see \cref{RMK-HIST}(1), the above has recently been done by Barton, Commelin \cite[Dfn 1.4.5]{Commelin} in their lecture notes on condensed mathematics. They used, essentially, the first line of the proof of \cref{cmp-final} as a definition.

\vspace{3pt}

\item[(3)] We will not alter our definition, but, in the remainder, often refer to the `topology of a compactological set' meaning the topology $\uptau$ from \cref{cmp-final}. Later, in \cref{DFN-COMP-4}, we will show that already the collection $\Comp$ of `plain sets' uniquely determines the $\uptau_K$ and from then on drop these from notation, cf.\ \cref{DFN-COMP-RMK}.\hfill\diam{}
\end{myitemize}
\end{rmk}

Let us now answer three natural questions. The first one is whether a compactology must always consist of \emph{all} sets that, in its own final topology, are compact. The answer to this is 'no' by the following familiar example:

\begin{ex}\label{EX-1} Let $(X,d)$ be an uncountable metric space. Then the collection $\Comp:= \{ K \subseteq X\:|\:K \text{ is countable and compact}\}$, where each $K$ is furnished with the topology $\uptau_K$ induced by the metric, defines a compactological set $(X, \Comp, (\uptau_K)_{K \in \Comp})$.\hfill\diam{}
\end{ex}

Although we did not include all compact sets above, the final topology as in \cref{cmp-final} nevertheless coincides with the $d$-topology we started with. The second question is whether this is always the case; again the answer is 'no':

\begin{ex}[Stempfhuber \protect{\cite[Ex 7.1.9]{Stempfhuber}}] Let $X:=\{ x \in \mathbb{R}^2\:|\:\left \lVert x \right \rVert \leqslant 1\}$ and denote by $d$ the Euclidean metric on $X$. Let $\Comp$ be the collection of closed subsets of finite unions of chords in $X$:\vspace{3pt}
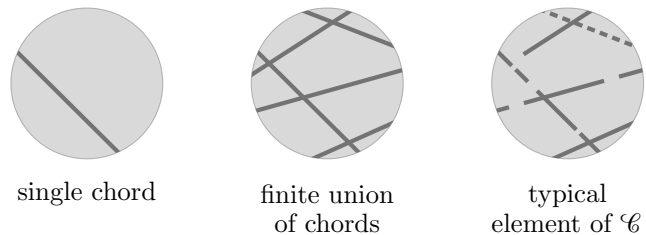
\begin{center}
\begin{tabular}{ccccc}
\begin{tikzpicture}
\draw[opacity=0.25] (1.25,1.25) circle [radius=10mm];
\fill[opacity=0.15](1.25,1.25) circle [radius=10mm];
%\draw[black] (0,0) rectangle (2.5,2.5);
  \begin{scope}
    \clip (1.25,1.25) circle [radius=10mm];
    \draw[ultra thick, black!55!white]   (0,2) -- node [] {} (2,0);
  \end{scope}
\end{tikzpicture}
&&
\begin{tikzpicture}
\draw[opacity=0.25] (1.25,1.25) circle [radius=10mm];
\fill[opacity=0.15](1.25,1.25) circle [radius=10mm];
%\draw (0,0) rectangle (2.5,2.5);
  \begin{scope}
    \clip (1.25,1.25) circle [radius=10mm];
    \draw[ultra thick, black!55!white]   (0,2) -- node [] {} (2,0);
    \draw[ultra thick, black!55!white]   (0,1.2) -- node [] {} (2.,2.5);
    \draw[ultra thick, black!55!white]   (0,0.8) -- node [] {} (2.5,1.5);
    \draw[ultra thick, black!55!white]   (0.5,0) -- node [] {} (2.5,0.9);
    \draw[ultra thick, black!55!white]   (0.2,2.5) -- node [] {} (2.5,1.6);
  \end{scope}
\end{tikzpicture}
&&
\begin{tikzpicture}
\draw[opacity=0.25] (1.25,1.25) circle [radius=10mm];
\fill[opacity=0.15](1.25,1.25) circle [radius=10mm];
%\draw (0,0) rectangle (2.5,2.5);
\begin{scope}
\clip (1.25,1.25) circle [radius=10mm];
\draw[ultra thick, black!55!white,shorten <=0cm ,shorten >=2.1cm]   (0,2) -- node [] {} (2,0);\draw[ultra thick, black!55!white,shorten <=0.8cm ,shorten >=1.8cm]   (0,2) -- node [] {} (2,0);\draw[ultra thick, black!55!white,shorten <=1.2cm,shorten >=0.9cm]   (0,2) -- node [] {} (2,0);\draw[ultra thick, black!55!white,shorten <=2cm,shorten >=0cm]   (0,2) -- node [] {} (2,0);
\draw[ultra thick, black!55!white,shorten <=0.8cm ]   (0,1.2) -- node [] {} (2.,2.5);
\draw[ultra thick, black!55!white]   (0.5,0) -- node [] {} (2.5,0.9);
\draw[ultra thick, black!55!white,shorten <=0cm ,shorten >=2.1cm]   (0,0.8) -- node [] {} (2.5,1.5);
\draw[ultra thick, black!55!white,shorten <=2cm ,shorten >=0cm]   (0,0.8) -- node [] {} (2.5,1.5);
\draw[ultra thick, black!55!white,shorten <=0.7cm ,shorten >=0.8cm]   (0,0.8) -- node [] {} (2.5,1.5);
\draw[ultra thick, black!55!white, densely dashed]   (0.2,2.5) -- node [] {} (2.5,1.6);
\end{scope}
\end{tikzpicture}\\[4pt]
\begin{minipage}{70pt}\centering \small  single chord\\ \phantom{C}\phantom{$\Comp$}\end{minipage} &&\begin{minipage}{70pt}\centering \small finite union\\[-2pt]of chords\end{minipage} &&\begin{minipage}{70pt}\centering \small  typical\\[-2pt]element of $\Comp$\end{minipage}\\
\end{tabular}
\end{center}\vspace{4pt}
Endowing each $K\in\Comp$ with the topology induced by $d$ yields a compactology. Using that there exists a $d$-dense subset of $X$ that contains no three collinear points, the final topology of the inclusions $K\hookrightarrow X$ turns out to be strictly finer than the topology induced by $d$. For details we refer to \cite[Ex 7.1.9]{Stempfhuber}.\hfill\diam{}
\end{ex}

Assume now that a compactological set is indeed defined by taking all compacts of a given topological space $(X,\upsigma)$. It will turn out that even then the final topology $\uptau$ might not coincide with $\upsigma$, see  \cref{rem:cgwh-cmop}(2). The third natural question, i.e., under which conditions on $(X,\upsigma)$ this actually happens, leads us to the notion of a compactly generated space. Moreover, as announced earlier, we define below the `correct' separation property for our setting.

\begin{dfn}[{Strickland \cite[Dfns 1.1 and 1.2]{CGWH}}]\label{CGWH-DFN} Let $X$ be a topological space.
\begin{enumerate}
\item  $X$ is \emph{weak Hausdorff} if for any compact space $S$ and any continuous map $f\colon S\rightarrow X$ the range $f(S)\subseteq X$ is closed.
\vspace{1pt}
\item $X$ is \emph{compactly generated} if a subset $F\subseteq X$ is closed iff $f^{-1}(F)$ is closed in $S$ for any compact space $S$ and any continuous map $f\colon S\rightarrow X$. 
\vspace{1pt}
\item $X$ is a \emph{cgwh-space} if it is compactly generated and weak Hausdorff.\hfill\diam{}
\end{enumerate}
\end{dfn}

We point out that the term `compactly generated' is not used consistently throughout the literature. However, under the additional assumption of being weak Hausdorff, at least two prominent definitions turn out to be equivalent: For a weak Hausdorff space $X$ to be compactly generated it suffices that condition \ref{CGWH-DFN}(2) holds for the inclusions of compact subspaces $f\colon S\hookrightarrow X$  only, see  \cite[Lem 1.4(c)]{CGWH}. For later use we note moreover that if $X$ is weak Hausdorff, $S$ is a compact topological space and $f\colon S\rightarrow X$ is continuous, then $f(S)\subseteq X$ is compact, see \cite[Lem 1.4(b)]{CGWH}.

\begin{dfn}\label{DFN-CMP-CAT} We denote by $\Cmp$ the \textit{category of compactological sets} in which a morphism  $f\colon C \to C'$ between $C = (X, \Comp, (\uptau_K)_{K\in\Comp})$ and $C'=(X', \Comp', (\uptau_{K'})_{K'\in\Comp'})$ is a map $f\colon X\rightarrow X'$ such that $f(K)\in \Comp'$ holds and $f|_K\colon(K,\uptau_K)\rightarrow(f(K),\uptau_{f(K)})$ is continuous for all $K \in \Comp$.  We denote by $\CGWH$ the category of cgwh-spaces with continuous maps as morphisms. \hfill\diam{}
\end{dfn}

For later use, let us mention the following consequence of \cref{cmp-final}: If $(X, \Comp, (\uptau_K)_{K\in\Comp})$ and $(X', \Comp', (\uptau_{K'})_{K'\in\Comp'})$ are compactological sets then a map $f\colon X\rightarrow X'$ such that $f(\Comp)\subseteq\Comp'$ is a morphism in $\Cmp$ if and only if $f$ is continuous with respect to the final topologies on $X$ and $X'$, respectively.

\smallskip

There are two forgetful functors $\Cmp\to\Top\to\Set$ which are both faithful. Given a compactological set $(X, \Comp, (\uptau_K)_{K\in\Comp})$, the first functor endows $X$ with the topology $\uptau$ from \cref{cmp-final} and forgets $\Comp$. Their composite admits a left adjoint $\Set\to\Cmp$ which assigns to any set the discrete topology upon it together with the compactology of finite sets.

\smallskip

We will now refine the forgetful relation between compactological sets and topological spaces to the pair of functors
\begin{equation*}
(-)_{\operatorname{top}}\colon\Cmp\rightarrow\CGWH\;\;\text{ and }\;\;\upiota\colon\CGWH\rightarrow\Cmp ,
\end{equation*}
where $(-)_{\operatorname{top}}$ is defined according to the previous paragraph and $\upiota$ endows a cgwh-space with the system of all compact subsets, each furnished with the induced topology, as a compactology.

\ThmAdjCGWHCmp

%\begin{thm}\label{CGWH-COMPL-ADJ} 
%    The two functors $\topuly-$ and $\upiota$ are well-defined and form an adjoint pair:
%    \begin{center}
%        \begin{tikzcd}[column sep=large, ampersand replacement=\&]
%            \CGWH
%            \arrow[r, "\upiota"{name=F}, bend left=17] \&
%            \Cmp.
%            \arrow[l, "(-)_{\operatorname{top}}"{name=G}, bend left=17]
%            %--- Adjunction Symbol
%            \arrow[phantom, from=F, to=G, "\dashv" rotate=-90]
%        \end{tikzcd}
%    \end{center}
%    Moreover, $\upiota$ is fully faithful, it holds that $(\upiota(-))_{\operatorname{top}} = \id_{\CGWH}$ and therefore we may read $\CGWH$ as a reflective subcategory of $\Cmp$. In particular, we may canonically understand the compact spaces $\CHaus\into\CGWH\into\Cmp$ as compactological sets endowed with the compactopology of all compact subsets, each carrying the  topology induced by $K$.
%\end{thm}

\begin{proof} \textcircled{1} Let $(X,\Comp,(\uptau_K)_{K\in\Comp})$ be a compactological set. We claim that $(X,\uptau)$ is weak Hausdorff: Let $S$ be a compact space and $f\colon S\rightarrow X$ be continuous. Then $f(S)\subseteq(X,\uptau)$ is quasi-compact and for arbitrary $K'\in\Comp$ we know that $K'\subseteq(X,\uptau)$ is compact due to \cref{cmp-final}. It follows that $f(S)\cap K'\subseteq (X,\uptau)$ is closed which establishes the claim. Next we show that $(X,\uptau)$ is compactly generated: Let $F\subseteq X$ be given and assume that $F\cap S\subseteq S$ is closed for any compact subspace $S\subseteq(X,\uptau)$. Then this holds in particular for all $S\in\Comp$, which means that $F\subseteq(X,\uptau)$ is closed. Thus, $(-)_{\operatorname{top}}$ is well-defined. That $\upiota$ is well-defined can be checked directly. 

\smallskip

\textcircled{2} To show that $(-)_{\operatorname{top}}$ is left adjoint to $\upiota$, let $X=(X,\Comp,(\uptau_K)_{K\in\Comp})$ be a compactological set and $Y=(Y,\upsigma)$ be a cgwh-space. We claim that
$$
\CGWH((X)_{\operatorname{top}},Y)\longrightarrow \Cmp(X,\upiota(Y)), \; f\mapsto f 
$$
is a bijection: As $f\colon(X)_{\operatorname{top}}\rightarrow Y$ is continuous and $Y$ is weak Hausdorff, $f$ maps every $K\in\Comp$ onto some compact subset of $Y$. The latter is by definition a member of the compactology of $\upiota(Y)$, whence it follows that $f\colon X\rightarrow\upiota(Y)$ is a morphism of compactological sets. Conversely, if a morphism  $f\colon X\rightarrow\upiota(Y)$ of compactological sets is given, then $f$ is continuous with respect to the (final) topologies. That this bijection is natural in $X$ and $Y$ follows from all composites being determined by their behavior on sets and the bijection acting as the identity on these underlying maps.

\smallskip

\textcircled{3} The compactology $\Comp$ of $\upiota(Y)$ consists of all compact subsets of $Y=(Y,\upsigma)$. As $Y$ is a cgwh-space, $\upsigma$ coincides with the final topology of all $(K,\upsigma|_{K})\into Y$ for $K\in\Comp$, cf.\ the remark just behind \cref{CGWH-DFN}. Thus $(\upiota(Y))_{\operatorname{top}}=Y$ holds.

\smallskip

\textcircled{4} For $Y,Y'\in\CGWH$ we use the adjunction formula from \textcircled{2} with $X=\upiota(Y)$ and \textcircled{3} to obtain
$$
\CGWH(Y,Y') = \CGWH((\upiota(Y))_{\operatorname{top}}, Y')\cong \Cmp(\upiota(Y),\upiota(Y')),
$$
establishing that $\upiota$ is fully faithful.
\end{proof}

Before investigating the categorical properties of the category $\Cmp$ in more detail, we add the following remarks.

\begin{rmk}\label{rem:cgwh-cmop}\begin{enumerate}\setlength{\itemindent}{-15pt}
        \item In general a compactological set bears more data than the associated cgwh-space. Consider, for instance, in Example \ref{EX-1} with $X=[0,1]$ and the Euclidian metric, as a second compactology $\Comp'$ the set of all compact subsets. Then $(X,\Comp',(\uptau_C)_{C\in\Comp'})$ and $(X,\Comp,(\uptau_C)_{C\in\Comp})$ cannot be isomorphic as compactological spaces, but their final topologies both coincide with the metric topology.% In particular $\upiota((X,\Comp',(\tau_C)_{C\in\Comp})_{\operatorname{top}})\not=(X,\Comp',(\tau_C)_{C\in\Comp})_{\operatorname{top}}$.
        
\setlength{\itemindent}{0pt}\vspace{2pt}

\item One can consider $\upiota$ as a functor from all weak Hausdorff spaces to compactological sets. The composition $(-)_{\operatorname{top}}\circ\upiota$ coincides then with what often is referred to as Kelley- or k-ification, see, e.g.\ \cite[p.\ 186]{MacLane}, \cite[Ex 4.3.10]{CatContext} or \cite[p.\ 690]{Schwede}. In particular, for any  weak Hausdorff space $(X,\upsigma)$ which is not compactly generated, the topology $\uptau$ as in  Proposition \ref{cmp-final} is necessarily strictly finer than $\upsigma$.

\vspace{2pt}

\item Buchwalter \cite[Chapter 3.3]{Buchwalter69} defined analogous functors on Hausdorff spaces; however, the above theorem shows that the cgwh-property is indeed sufficient for our purposes. Waelbroeck \cite[p.~44]{Waelbroeck} stated that the same construction may be applied to $T_1$-spaces, where however finite unions of Hausdorff subsets need not be Hausdorff.
\end{enumerate}
\end{rmk}

\smallskip

Compactological sets are a way of encoding simultaneously notions of being `close' and of being `small'. The latter word makes even more sense in the Waelbroeck approach, cf.~Remarks \ref{RMK-HIST} and \ref{RMK-COMMELIN}, where arbitrary subsets of small sets are again small. While it is convenient to have access to both $\uptau$ and $\Comp$ explicitly, both structures are indeed encoded in the system $\Comp$ of `plain sets' only. The following proposition makes this idea precise.

\begin{prop}\label{DFN-COMP-4} Let $(X, \Comp, (\uptau_K)_{K\in\Comp})$ be a compactological set. Then $\Comp\subseteq\Pow(X)$ has the following properties:
    \begin{enumerate}
        \item $\Comp$ contains the empty set.
        \item $\Comp$ is a covering of $X$.
        \item $\Comp$ is stable under finite unions.
        \item $\Comp$ is stable under arbitrary intersections over non-empty index sets.
        \item For any two points $x\neq y$ in $X$ and every $K\in\Comp$ that contains both, there exist $K_1, K_2\in\Comp$ such that $K=K_1\cup K_2$ and $x\notin K_1$ and $y\notin K_2$.
        \item Every collection $(K_i)_{i\in I}\subseteq\Comp$ with empty intersection admits a finite subcollection $(K_i)_{i\in I'}$ with empty intersection.
    \end{enumerate}
    Furthermore, if some $\Comp\subseteq\Pow(X)$ satisfies (1)--(6), then there is a unique system of topologies such that $(X,\Comp,(\uptau_K)_{K\in\Comp})$ is a compactological set. This system is given by $\uptau_K := \{K\setminus K' \mid K'\in\Comp\}$. 
\end{prop}

\begin{proof} We split the proof into several parts.

\smallskip

\textcircled{1} A compactological set satisfies \ref{DFN-COMP-4}(1)\hspace{1pt}--\hspace{1pt}(6): We have $\Comp\not=\varnothing$ by  the general assumption in \cref{DFN-COMP-0}. If $X\not=\varnothing$, then \ref{DFN-COMP-0}(4) implies $\varnothing\in\Comp$. Otherwise $X=\varnothing$ and then $\Comp\not=\varnothing$ enforces $\Comp=\Pow(\varnothing)=\{\varnothing\}$. This shows \ref{DFN-COMP-4}(1); \ref{DFN-COMP-4}(2) coincides with \ref{DFN-COMP-0}(1) and \ref{DFN-COMP-4}(3) follows from \cref{FIRST-LEM}. For \ref{DFN-COMP-4}(4) let $(K_i)_{i\in I}$ be given, pick $i_0\in I$ and denote by $\uptau$ the final topology as in \cref{cmp-final}. By the latter proposition all $K_i$, and thus also their intersection $K:=\fcap_{i\in I}K_i$, are in particular closed subsets of $(X,\uptau)$. Since $\uptau_{K_{i_{\scalebox{0.5}{$0$}}}}=\uptau|_{K_{i_{\scalebox{0.5}{$0$}}}}$ holds, we get that     $K\subseteq(K_{i_0},\uptau_{K_{i_{\scalebox{0.5}{$0$}}}})$ is closed. Employing \ref{DFN-COMP-0}(4) we conclude $K\in\Comp$. Condition \ref{DFN-COMP-4}(5) is straightforward by the $K$ being Hausdorff. It remains to check \ref{DFN-COMP-4}(6): Given $(K_i)_{i\in I}\subseteq\Comp$ with empty intersection we notice that for $I=\varnothing$ the claim is trivial and thus we may pick $i_0\in I$. By arguments as used above we get that $K_i\cap K_{i_0}\subseteq(K_{i_0},\uptau_{K_{i_{\scalebox{0.5}{$0$}}}})$ is closed for any $i\in I$.  It follows that $U_i:=K_{i_0}\backslash{}K_i\subseteq(K_{i_0},\uptau_{K_{i_{\scalebox{0.5}{$0$}}}})$ is open and since $\fcup_{i\in I}U_i=K_{i_0}$ and $K_{i_{0}}$ is compact we find a finite subset $I'\subseteq I$ such that $K_{i_0}=\bigcup_{i\in I'}U_i=K_{i_0}\backslash\bigcap_{i\in I'}K_i$ holds. From the latter we conclude $K_{i_0}\cap\bigcap_{i\in I'}K_i=\varnothing$.

\smallskip

\textcircled{2} The $\uptau_K=\{K\backslash K'\}$ are compact topological spaces: Employing \ref{DFN-COMP-4}(1), \ref{DFN-COMP-4}(3) and \ref{DFN-COMP-4}(4), it is straightforward to verify that $\uptau_K$ is indeed a topology on $K$. Using \ref{DFN-COMP-4}(5) one can see that $\uptau_K$ is Hausdorff. To see that $(K,\uptau_K)$ is quasi-compact let $(U_i)_{i\in I}$ be an open cover. By definition $U_i=K\backslash K_i$ holds with suitable $K_i\in\Comp$ for $i\in I$. Then $K\cap\fcap_{i\in I}K_i=\varnothing$ and applying \ref{DFN-COMP-4}(6) to the latter system yields a finite subset $I'\subseteq I$ with $K\cap\bigcap_{i\in I'}K_i=\varnothing$. This implies $K=\bigcup_{i\in I'}U_i$.

\smallskip

\textcircled{3} For $K\in\Comp$ we have $\{F\subseteq K\:|\:F\text{ is }\uptau_K\text{-closed}\}=\{K\cap K'\:|\:K'\in\Comp\}=\{K'\in\Comp\:|\:K'\subseteq K\}$: The first equality follows from the explicit description of $\uptau_K$. The second follows from \ref{DFN-COMP-4}(4).

\smallskip

\textcircled{4} The triple $(X,\Comp,(\uptau_K)_{K\in\Comp})$ is a compactological set: By \textcircled{2} the $(K,\uptau_K)$ are compact and the first two conditions in \cref{DFN-COMP-0} are satisfied by assumption. Condition \ref{DFN-COMP-0}(3) and \ref{DFN-COMP-0}(4) follow from \textcircled{3}.

\smallskip

\textcircled{5} The $\uptau_K$ are unique: Let $(\upsigma_K)_{K\in\Comp}$ be some system of topologies such that $(X,\Comp,(\upsigma_K)_{K\in\Comp})$ is a compactological set. Fix $K\in\Comp$. A subset $F\subseteq K$ is then $\upsigma_K$-closed iff it is $\upsigma_K$-compact iff it belongs to $\Comp$. The first equivalence holds as $(K,\upsigma_K)$ is compact and the second follows from \ref{DFN-COMP-0}(3) and \ref{DFN-COMP-0}(4).
\end{proof}

In many cases the following lemma provides a shortcut for applying Proposition \ref{DFN-COMP-4} if we seek to verify that a collection of sets $\Comp$ defines a compactology.

\begin{lem}\label{LEM-SHORTCUT} Let $X\not=\varnothing$ and assume that $\Comp\subseteq\Pow(X)$ satisfies \ref{DFN-COMP-4}(1)--\ref{DFN-COMP-4}(4). Putting $\uptau_K:=\{K\backslash K'\ |\ K'\in\Comp\}$ defines topologies on $K\in\Comp$. Letting $\uptau$ be the final topology of the inclusions $\{(K,\uptau_K)\hookrightarrow X\ |\ K\in\Comp\}$, the following holds true:
\begin{enumerate}
\item[(1)] Let $K\in\Comp$. Then $\{F\!\subseteq\!K\,|\,F\text{ is }\uptau_K\text{-closed\hspace{1pt}}\}=\{K\cap K'\,|\,K'\in\Comp\}=\{K'\!\in\!\Comp\,|\,K'\subseteq K\}$.\vspace{2pt}
\item[(2)] For every $K\in\Comp$ we have $\uptau_K = \uptau|_K$,\vspace{2pt}
\item[(3)] If $F\subseteq(X,\uptau)$ is closed and contained in some $K\in\Comp$, then we have $F\in\Comp$.\vspace{2pt}
\item[(4)] If $(K,\uptau_K)$ is compact for every $K\in\Comp$, then $\Comp$ satisfies \ref{DFN-COMP-4}(5) and \ref{DFN-COMP-4}(6).
\end{enumerate}
\end{lem}
\begin{proof} We noted already in part \textcircled{2} of the proof of  \cref{DFN-COMP-4} that actually \ref{DFN-COMP-4}(1), \ref{DFN-COMP-4}(3) and \ref{DFN-COMP-4}(4) are enough for the $\uptau_K$ to be topologies. Moreover, we noted in \textcircled{3} of the same proof that the equalities in (1) hold and indeed we did only use \ref{DFN-COMP-4}(1)--\ref{DFN-COMP-4}(4) for this. Now, (1) together with \ref{DFN-COMP-4}(4) yields (2). Furthermore, (3) is a consequence of (1) and (2).

\smallskip

To show (4) let finally all $(K,\uptau_K)$ be compact. Using the description of the $\uptau_K$ and the fact that they are Hausdorff, it can be checked directly that \ref{DFN-COMP-4}(5) holds. Moreover, $\uptau_K$ being quasi-compact implies \ref{DFN-COMP-4}(6) in a straightforward manner.
\end{proof}

\begin{rmk}\label{DFN-COMP-RMK} In the remainder we will refer to compactological sets just as $C=(X,\Comp)$ knowing that the topologies $\uptau_K$ may at any time be recovered via \cref{DFN-COMP-4}. Moreover, we will often use the above proposition and lemma to verify that some collection $\Comp$ actually is a compactology.\hfill\diam{}
\end{rmk}

\smallskip

In order to connect compactological sets with condensed sets in Section \ref{SEC:Equiv}, we need to deal with certain set-theoretic issues. For this we fix an uncountable strong limit cardinal $\upkappa$. Prototypically, $\upkappa$ can be taken to be of the form $\beth_\uplambda$ where $\uplambda$ is some appropriately chosen regular cardinal. In this case $\operatorname{cof}(\upkappa) = \uplambda$.

\smallskip

\begin{dfn}By a \emph{$\upkappa$-small compact space} we mean a compact space $K$ such that $|K|<\upkappa$. We say that a topological space is \emph{$\upkappa$-compactly generated}, resp.\ \emph{$\upkappa$-weak Hausdorff}, if the conditions in \cref{CGWH-DFN}(2), resp.\ \cref{CGWH-DFN}(1), hold for all $\upkappa$-small compact spaces $S$. Further, we define a \emph{$\upkappa$-cgwh space} to be a space that is both $\upkappa$-compactly generated and $\upkappa$-weak Hausdorff. We denote by $\CGWH_\upkappa$ the full subcategory of $\CGWH$ formed by the $\upkappa$-cgwh spaces. Finally, we define a \emph{$\upkappa$-compactological set} to be a compactological set $(X, \Comp)$ such that every set $K\in\Comp$ is $\upkappa$-small. The category $\Cmp_\upkappa$ is defined accordingly.\hfill\diam{}
\end{dfn}

\begin{rmk}\label{rem:cmp-size}
    \begin{enumerate}
        \setlength{\itemindent}{-15pt} 
        \item Notice that a $\upkappa$-compactly generated space is weak Hausdorff if and only if it is $\upkappa$-weak Hausdorff. In our definition of $\upkappa$-cgwh spaces above one thus may replace `$\upkappa$-weak Hausdorff' just by `weak Hausdorff'.
        %The weaker condition that a $\kappa$-cgwh space only needs to be $\kappa$-weak Hausdorff is inconsequential, as any $\kappa$-compactly generated space is weak Hausdorff if and only if it is $\kappa$-weak Hausdorff.\NOTE{Sven}{Besser formulieren?}

        \setlength{\itemindent}{0pt}\vspace{2pt}
        \item Below we will always distinguish between the general and the restricted categories via the index $\upkappa$. 
        %Often we make no explicit reference to the cardinal $\kappa$ and understand a statement to apply both to the general and to the restricted case.
        Due to set-theoretical considerations, not every statement can be transferred from $\Cmp_\upkappa$ to $\Cmp$. Assuming the existence of a Grothendieck universe $V$ and letting a small set refer to an element of $V$, $\upkappa$ can be chosen to be $|V|$, which corresponds to the unrestricted case within stronger axiomatic foundations. Under these assumptions $\Cmp$ is just a special case of $\Cmp_\kappa$ and hence all results transfer. In particular one can work directly with $\Cmp$ without problems.
        \vspace{2pt}
        \item The inclusion $\Cmp_\upkappa\into\Cmp_{\upkappa'}$ for $\upkappa<\upkappa'$ admits a right adjoint, which forgets all elements of the compactology which are not $\upkappa$-small. There is an analogous adjunction for cgwh-spaces. Furthermore, both the category $\Cmp$ and $\CGWH$ are the totally ordered colimit over all $\upkappa$-restricted subcategories and every inclusion admits a corresponding adjunction in the same manner. In this sense, the unrestricted case can be understood to be the union of all restrictions and many results extend from all restrictions to the unrestricted category.
        \vspace{2pt}
        \item With the above in place, \cref{CGWH-COMPL-ADJ} transfers verbatim to an adjunction between $\kappa$-compactological sets and $\upkappa$-cgwh spaces.
    \end{enumerate}
\end{rmk}

Before introducing categories of condensed sets in the next section, we give an overview of limits and colimits of compactological sets, in particular investigating products, coproducts, quotients and subobjects. Corresponding results in the Waelbroeck setting, cf.\ \cref{RMK-HIST}, can be found in \cite[Chapter 7]{Stempfhuber}.

\begin{prop}\label{prop:cmp-prod} 
    Let $I$ be a set and let $(C_i)_{i\in I}$ be a family of $\upkappa$-compactological sets $C_i=(X_i, \Comp_i)$. 
    \begin{enumerate}
        \item Denote by $\overline{(-)}^{\,\cup_{\operatorname{fin}}}$ the closure under finite unions. The coproduct $\fcoprod_{i\in I} C_i$ of $(C_i)_{i\in I}$ in both $\Cmp$ and $\Cmp_\upkappa$ exists and is given by
            \[
                 \Bigl(X := \fcoprod_{i\in I} X_i, \Comp := \overline{\mathop{\fsqcup}_{i\in I} \Comp_i}^{\,\cup_{\operatorname{fin}}}\Bigr).
            \]
            The associated topology of $\fcoprod_{i\in I} C_i$ is the usual coproduct topology.\vspace{3pt}
        \item For $\Born\subseteq\Pow(X)$ denote by $\overline{\Born}^{\,\cap\hspace{1pt}\cup_{\operatorname{fin}}}$ the collection of sets that can be written as  intersections over a non-empty collection of finite unions of elements of $\Born$. The product $\fprod_{i\in I} C_i$ of $(C_i)_{i\in I}$ in $\Cmp$ exists and is given by 
            \[
                \Bigl(X := \fprod_{i\in I} X_i, \Comp := \overline{\Bigl\{\fprod_{i\in I} K_i \Bigm| \forall\:i\in I: K_i \in \Comp_i\Bigr\}}^{\,\cap\hspace{1pt}\cup_{\operatorname{fin}}}\Bigr).
            \]
            If the index set $I$ is finite, the topology on $X$ induced by $\Comp$ is the  topology of the product computed in $\CGWH$ \cite[Dfn 2.3 and Prop 2.4]{CGWH}.\vspace{3pt}
%\blue{For bornological spaces the products $\prod_i K_i$ form a basis of the product bornology in the sense that every bounded set arises as a subset of a $\prod_i K_i$. We also don't see how we could \textbf{add} sets to the system of products and guarantee their compactness. Idea: $\Comp$ should be the set of all subsets of the $\prod_i K_i$'s which are relatively compact by Tychonoff, cf.~Hogbe-Nlend p.~30. Wouldn't it be off if the forgetful functor from compactological sets to bornological sets wouldn'd commute with products?}

\item The product of $(C_i)_{i\in I}$ for $I$ as above exists in $\Cmp_\upkappa$ as well and it is obtained by computing the product in $\Cmp$ and retaining only the $\upkappa$-small elements of the compactology. In particular, it agrees with the above whenever the involved spaces and the index set are sufficiently small in comparison to $\upkappa$. Among others this includes the case when both $|I|$ and $\sup_{K\in \bigcup_{i\in I}\Comp_i} |K|$ are strictly smaller than $\upkappa$, or, when $|I|<\operatorname{cof}(\upkappa)$.% is smaller than the cofinality of $\kappa$, as $\kappa$ was assumed to be a strong limit cardinal.
\end{enumerate}
\end{prop}

\begin{proof}(1) We first check that $(X,\Comp)$ defines a compactological set via Proposition \ref{DFN-COMP-4}. Indeed, the conditions \ref{DFN-COMP-4}(1)--\ref{DFN-COMP-4}(3) are straightforward. To check \ref{DFN-COMP-4}(4) let $(K^{\scriptscriptstyle(j)})_{j\in J}$ be given with $j_0\in J$. Then
$$
\fcap_{j\in J}K^{\scriptscriptstyle(j)}=\bigl(K_{1}^{\scriptscriptstyle(j_{\scalebox{0.5}{$0$}})}\times\{i_{1}^{\scriptscriptstyle(j_{\scalebox{0.5}{$0$}})}\}\cup\cdots\cup K_{n_{j_{\scalebox{0.5}{$0$}}}}^{\scriptscriptstyle(j_{\scalebox{0.5}{$0$}})}\times\{i_{n_{j_{\scalebox{0.5}{$0$}}}}^{\scriptscriptstyle(j_{\scalebox{0.5}{$0$}})}\}\bigr)\cap\fcap_{j\in J\backslash\{j_0\}}\hspace{-8pt}K^{\scriptscriptstyle(j)}
$$
from where we see that the intersection consists of a union of at most $n_{j_0}$-many sets, each of which is an intersection in some $\Comp_i$. Since we established \ref{DFN-COMP-4}(3) already, we may apply the latter to conclude that \ref{DFN-COMP-4}(4) holds. Next we use Lemma \ref{LEM-SHORTCUT}(4) to verify \ref{DFN-COMP-4}(5) and \ref{DFN-COMP-4}(6). For this it is enough to observe that any element
$$
K=\fcup_{j=1}^nK_{j}\times\{i_{j}\}\in\Comp,
$$
where w.l.o.g.\ $i_{j_1}\not=i_{j_2}$ for $j_1\not=j_2$ may be assumed, can be endowed with a compact topology coming from the $K_j\in\Comp_{i_j}$. A direct calculation shows that the latter coincides with the topology $\uptau_K$ as in Lemma \ref{LEM-SHORTCUT}. Thus $(X,\Comp)$ is a compactological set.
%
%For \ref{DFN-COMP-4}(5) let distinct points
%$$
%(x,\mu),(y,\nu)\in \fcup_{j=1}^nK_{j}\times\{i_{j}\}\in\Comp
%$$
%be given. If $\mu=\nu$, then the desired decomposition can be found by using that \ref{DFN-COMP-4}(5) holds for $\Comp_\mu$. If $\mu\not=\nu$, then put
%$$
%K^{(1)}:=\fcup_{\stackrel{\scriptstyle j=1}{i_j\not=\mu}}^nK_{j}\times\{i_{j}\} \;\text{ and }\;K^{(2)}:=\fcup_{\stackrel{\scriptstyle j=1}{i_j\not=\nu}}^nK_{j}\times\{i_{j}\}.
%$$
%For \ref{DFN-COMP-4}(6) let $(K^{\scriptscriptstyle(j)})_{j\in J}$ with $j_0\in J$ have empty intersection. As above we see that this intersection is in fact a finite union of intersections in certain $\Comp_i$'s. But now, by assumption each of these intersections is empty and we may use that every $\Comp_i$ enjoys \ref{DFN-COMP-4}(6). 
%
On $X$ we can now consider on the one hand the topology induced by $\Comp$ in the sense of \cref{DFN-COMP-4} and on the other hand we can use \cref{DFN-COMP-4} to get a topology on every $X_i$ first and then endow $X$ with the corresponding coproduct-topology. Using that $K_1\times\{i_1\}$ and $K_2\times\{i_2\}$ are either in the same $\Comp_i$ or disjoint, it follows however that both topologies coincide. From this we see that the canonical maps $X_i\rightarrow X$ are continuous. As they, by construction, map elements of $\Comp_i$ onto elements of $\Comp$, they are morphisms in $\Cmp$. Let finally $f_i\colon (X_i,\Comp_i)\rightarrow(Y,\Eomp)$ be morphisms into some other compactological set. Regarded only as topological spaces, $X$ is the coproduct of the $X_i$, and we thus get a unique continuous map
$$
f\colon\fcoprod_{i\in I}X_i\rightarrow Y,
$$
making the coproduct diagram commutative. Using that $\Eomp$ is closed under finite unions we see that $f(\Comp)\subseteq\Eomp$ holds.

\medskip

(2) In the steps \textcircled{1}\hspace{1pt}--\hspace{1pt}\textcircled{3} below we will show that $\Comp$ is a compactology. In \textcircled{4} we will then establish that $(X,\Comp)$ satisfies the universal property of a product in $\Cmp$ and in \textcircled{5} we will show the last statement of (2) on finite index sets. Throughout we will use the abbreviation $\Qomp=\bigl\{\prod_{i\in I}K_i\:|\:\forall\:i\in I\colon K_i\in\Comp_i\bigr\}$ for the `generator' of $\Comp$.

\smallskip

\textcircled{1} We observe that $\Comp=\overline{\Qomp}^{\,\cap\hspace{1pt}\cup_{\operatorname{fin}}}$ is by definition closed under arbitrary intersections over non-empty index sets. By passing to complements, one can verify that $\Comp$ is closed under finite unions. Moreover, we get $\varnothing\in\Comp$ as a union over the empty index set and see that $\Comp$ covers $X$ by using that each $\Comp_i$ covers $X_i$. Let  $\uptau_K$ for $K\in\Comp$ and $\uptau$ be as in \cref{LEM-SHORTCUT} and put
$$
\Domp:=\bigl\{F\subseteq X\:|\:F\subseteq(X,\uptau)\text{ is closed and }\exists\:Q\in\Qomp\colon F\subseteq Q\bigr\}.
$$
A direct calculation shows that (a) $\Domp$ is closed under finite unions as well as under intersection over arbitrary non-empty index sets. Since $\Qomp\subseteq\Comp$, \cref{LEM-SHORTCUT}(1) implies that (b) $\Qomp\subseteq\Domp$ holds. As $\Comp$ is however by construction the smallest subset of $\Pow(X)$ satisfying (a) and (b), we infer $\Comp\subseteq\Domp$.

\smallskip

\textcircled{2} We fix $Q=\prod_{i\in I}K_i\in\Qomp$, denote by $\upsigma$ the product topology with respect to the topologies $\uptau_{K_i}=\uptau_i|_{K_i}$ and claim that $\upsigma=\uptau|_Q=\uptau_Q$ holds. Notice that the second equality holds by \cref{LEM-SHORTCUT}(2) and only the first has to be established here. To achieve this, observe first that $A\subseteq(Q,\upsigma)$ is closed iff
$$
\exists\:A_{jk}^{\scriptscriptstyle(i)}\subseteq K_i\text{ closed with } A_{jk}^{\scriptscriptstyle(i)}\not=K_i \text{ for at most one }i\in I\text{ s.th.: }A=\fcap_{j\in J}\fcup_{k=1}^{n_j}\fprod_{i\in I}A_{jk}^{\scriptscriptstyle(i)}
$$
holds. Next, observe
\begin{eqnarray*}
\bigl\{F\:|\:F\subseteq(Q,\uptau_Q)\text{ closed}\bigr\} & = & \bigl\{Q\cap K\:|\:K\in\Comp\bigr\}\\[4pt]
& = & \overline{\bigl\{Q\cap Q'\:|\:Q'\in\Qomp\bigr\}}^{\,\cap\hspace{1pt}\cup_{\operatorname{fin}}}\\[2pt]
& = & \overline{\Bigl\{\fprod_{i\in I}A_{jk}^{\scriptscriptstyle(i)}\:\big|\:\forall\:i\in I\colon A_{jk}^{\scriptscriptstyle(i)}\subseteq (K_i,\uptau_{K_i})\text{ closed}}\\[-5pt]
&  & \phantom{xx\hspace{25pt}xx} \overline{\text{and }A_{jk}^{\scriptscriptstyle(i)}\not= K_i \text{ for at most one }i\in I\Bigr\}}^{\,\cap\hspace{1pt}\cup_{\operatorname{fin}}}.
\end{eqnarray*}
Indeed, the first equality follows as $Q\in\Comp$. For the second equality `$\subseteq$' follows from the definition of $\Comp$, while `$\supseteq$' holds as all $Q\cap Q'$, and then even all members of the relevant set, are closed in $(Q,\uptau_Q)$. Since on the one hand all $Q\cap Q'$ belong to $\Qomp$ and on the other hand all $\prod_{i\in I}A_{jk}^{\scriptscriptstyle(i)}$ belong to $\Qomp$ we get also the last equality.

\smallskip

\textcircled{3} Let $K\in\Comp$ be given. By \textcircled{1} there exists $Q\in\Qomp$ with $K\subseteq Q$. \cref{LEM-SHORTCUT}(1) implies that $K\subseteq(Q,\uptau_Q)$ is closed. Using \textcircled{2} we see that $(Q,\uptau_Q)=(Q,\upsigma)$ and the latter is compact by Tychonoff's theorem. Thus, $(K,\uptau_Q|_K)=(K,\uptau|_K)=(K,\uptau_K)$ is compact and it follows from \cref{LEM-SHORTCUT}(4) that now all six conditions in \cref{DFN-COMP-4} hold.

\smallskip

Before we continue with the proof, let us remark that it follows from \cref{LEM-SHORTCUT}(3), now that we know that $\Comp$ is a compactology, that $\Comp=\Domp$ holds.

\smallskip

\textcircled{4} We first verify that for $j\in I$ the projection
$$
p_j\colon\fprod_{i\in I}X_i\rightarrow X_j
$$
is a morphism in $\Cmp$. In order to see that $p_j$ is continuous, let $K\in\Comp$ be arbitrary and consider $p_j|_K$. Then $K\subseteq Q:=\prod_{i\in I}K_i$ with $K_i\in\Comp_i$. By \textcircled{3} we know that $(K,\uptau_K)\subseteq(Q,\uptau|_Q)\in\Qomp$ holds with certain $K_i\in\Comp_i$ where $\uptau$ coincides with the product topology on $Q$. But then the projection $(Q,\uptau|_Q)\rightarrow(K_j,\uptau_{K_j})$ is continuous and the latter is a topological subspace of $(X_j,\uptau_j)$. In order to check $p_j(\Comp)\subseteq\Comp_j$ we take $K$ and choose $Q$ as above. Then $p_j(K)\subseteq (K_j,\uptau_{K_j})$ is compact and thus $p_j(K)\in\Comp_j$.

\smallskip

Let now $f_i\colon(Y,\Eomp)\rightarrow(X_i,\Comp_i)$ be morphisms from some other compactological set. Regarded only as sets, we get a unique map
$$
f\colon Y\rightarrow\fprod_{i\in I}X_i
$$
which, for given $K\in\Eomp$, satisfies $f(K)\subseteq Q:=\prod_{i\in I}K_i\in\Comp$ with certain $K_i\in\Comp_i$. By using similar arguments as above it is enough to show that $f\colon(K,\uptau_K)\rightarrow(Q,\uptau|_Q)$ is continuous. As $\uptau|_Q$ coincides with the product topology on $Q$, this, however, follows by using the latter's universal property.

\medskip  

\textcircled{5} Consider now $I=\{1,2\}$. We claim that the topology of the product $\topuly{X_1,\Comp_1}\times\topuly{X_2,\Comp_2}$ in $\CGWH$ is final with respect to the elements $K\in\Comp$ of the compactology on $X=X_1\times X_2$. As any $K\in\Comp$ is a subspace of some product $K_1\times K_2$ for $K_1\in\Comp_1$ and $K_2\in\Comp_2$ this is equivalent to finality with respect to these.
To see that this holds, observe that any space $Y$ has the final topology with respect to a jointly surjective family $(Z\to Y)_{Z\in\mathscr Z}$ of maps if and only if $\coprod_{Z\in\mathscr Z} Z \to Y$ is a quotient map. We show this for
\[
    \fcoprod_{(K_1, K_2)\in\Comp_1\times\Comp_2} K_1\times K_2 \onto \topuly{X_1, \Comp_1}\times \topuly{X_2, \Comp_2}.
\]
It is immediate that this map corresponds to the categorical product of
\[
    \fcoprod_{K_1\in\Comp_1} K_1 \onto \topuly{X_1, \Comp_1} \;\;\;\text{ and }\;
    \fcoprod_{K_2\in\Comp_2} K_2 \onto \topuly{X_2, \Comp_2}.
\]
Both of these are quotient maps by definition and finite products of quotient maps between compactly generated spaces are quotient maps according to \cite[Prop 2.20]{CGWH}.

\medskip

(3) By \cref{rem:cmp-size}(3) the inclusion $\upiota\colon\Cmp_\upkappa\into\Cmp$ admits a right adjoint $\uppi\colon\Cmp\to\Cmp_\upkappa$, which necessarily preserves limits. Furthermore $\Cmp_\upkappa\into\Cmp\to\Cmp_\kappa$ is the identity. Thus for any diagram $F\colon \catI\to\Cmp_\upkappa$, $\lim F$ is given by $\uppi(\lim \upiota(F))$, as claimed. 

\smallskip

If $|I|$ and $\sup_{K\in \bigcup_{i\in I}\Comp_i} |K|$ are strictly smaller than $\upkappa$ it follows from cardinal arithmetic that
$$
\sup_{K\in\Comp} |K|\hspace{3pt} =\hspace{-3pt} \sup_{\scalebox{0.6}{$\displaystyle(K_i)\hspace{-1pt}\in\hspace{-1pt}\fprod_{i\in I} \Comp_i$}}\hspace{3pt}\fprod_{i\in I} |K_i| \hspace{3pt}\leqslant\hspace{2pt} \Bigl(\hspace{1pt}\sup_{\scalebox{0.6}{$\displaystyle K\hspace{-1pt}\in\hspace{-1pt}\fcup_{i\in I} \Comp_i$}} |K|\Bigr)^{|I|} \leqslant \hspace{3pt}2^{|I|\hspace{1.5pt}\cdot\hspace{1.5pt}\sup_{K\in \bigcup_{i\in I} \Comp_i} |K|}\hspace{3pt}=\hspace{3pt}2^\uplambda
$$
for some $\uplambda<\upkappa$. That $\upkappa$ is a strong limit cardinal then implies that the elements of the product compactology are again bounded by $\upkappa$ and hence the products in $\Cmp$ and $\Cmp_\upkappa$ agree. The case of $|I|<\operatorname{cof}(\upkappa)$ follows by a similar argument, i.e., bounding $\bigl|\prod_{i\in I} K_i\bigr|$ by $\bigl(\sup_{i\in I} |K_i|\bigr)^{|I|}$, showing that every individual element of $\Comp$ will be bounded by $\upkappa$ even when the supremum $\sup_{K\in\Comp} |K|$ might not be.
\end{proof}

\begin{prop}\label{prop:cmp-equ} 
    Let $f, g\colon C\to D\in\Cmp_\kappa$ be any two morphisms of $\upkappa$-compactological sets where $C=(X_C, \Comp_C)$ and $D=(X_D, \Comp_D)$.\vspace{3pt}
    \begin{enumerate}
        \item The equalizer $\operatorname{Eq}(f, g)$ of $f$ and $g$ in $\Cmp_\kappa$ and $\Cmp$ exists and is given by
            \[
                \phantom{XXX}\bigl(X := \bigl\{x\in X_C\mid f(x) = g(x)\bigr\},\ \Comp := X\cap\Comp_C = \bigl\{K\in\Comp_C\mid K\subseteq X\bigr\}\bigr).
            \]
            The associated cgwh-space $\topuly{\operatorname{Eq}(f, g)}$ agrees with the equalizer of $\topuly f$ and $\topuly g$ in $\CGWH$.\vspace{3pt}
        \item The coequalizer $\operatorname{Coeq}(f, g)$ of $f$ and $g$ in $\Cmp_\upkappa$ and $\Cmp$ exists and is given by
            \[
            \phantom{XXX}\bigl(X := X_D/E,\ \Comp := \uppi(\Comp_D)\bigr),
                %\phantom{XXX}\bigl(X := X_D \Big/\ \overline{\bigl\{(f(x), g(x)) \mid x\in X_C\bigr\}},\ \Comp := \uppi(\Comp_D)\bigr)
            \]
            where $E\subseteq X_D\times X_D$ is the smallest closed equivalence relation such that $E$ contains all pairs $(f(x),g(x))$ for $x\in X_C$, where the product carries the topology induced by the product compactology, cf.\ \cref{prop:cmp-prod}(2).
            %where $\overline{M}$ denotes the closure of a set $M\subseteq X_D\times X_D$ to the smallest equivalence relation containing $M$ that is closed as a subset of $D\times D$. 
            The map $\uppi\colon X_D\to X$ is the canonical projection. The associated cgwh-space $\topuly{\operatorname{Coeq}(f, g)}$ agrees with the coequalizer of $\topuly f$ and $\topuly g$ in $\CGWH$.
    \end{enumerate}
\end{prop}

\begin{proof} (1) According to \cref{CGWH-COMPL-ADJ} the topologies $\uptau_C$ and $\uptau_D$ given by $\Comp_C$ and $\Comp_D$ are cgwh. Using \cite[Prop 2.14]{CGWH} we obtain that the diagonal $X_D\cong \Delta_D \into X_D\times X_D$ is closed as a subspace of the compactly generated product and hence by continuity $X = (f, g)^{-1}(\Delta_D)$ is a closed subset of $(X_C,\uptau_C)$. We claim that, in general, for any closed subset $X\subseteq(X_C,\uptau_C)$, the \emph{subset compactology}
$$
\bigl(X,\ \Comp:=\bigl\{K\in\Comp_C\ | \ K\subseteq X\bigr\}=\bigl\{K\cap X\ | \ K\in\Comp_C\bigr\}\bigr)
$$
is a compactology. Indeed, \ref{DFN-COMP-4}(1)--\ref{DFN-COMP-4}(4) are satisfied. As
$$
\{K\cap K'\ | \ K'\in\Comp\}=\{K\cap K'\ | \ K'\in\Comp_C\}
$$
holds for $K\in\Comp$ the topologies $\uptau_K$ induced by $\Comp$ resp.~$\Comp_C$ coincide and we may apply \cref{LEM-SHORTCUT}(4) to get \ref{DFN-COMP-4}(5) and \ref{DFN-COMP-4}(6). Moreover, using the consistency of the topologies $\uptau_K$ it follows that the topology $\uptau$ induced by $\Comp$ on $X$ equals the restriction of $\uptau_C$ to $X$; in particular, the inclusion $\operatorname{eq}\colon X\hookrightarrow X_C$ is continuous and by definition $\operatorname{eq}(\Comp)\subseteq\Comp_C$. From here it is straightforward to check that $\operatorname{eq}$ satisfies the universal properties in $\Cmp$ and in $\CGWH$.

\medskip

(2) We first notice that $E$ can explicitly be described as the intersection over all equivalence relations that are closed and contain $\{(f(x),g(x))\:|\:x\in X_C\}$. By \cite[Prop 2.1 and Cor 2.21]{CGWH} the quotient topology on $X_D/E$ is cgwh. We write $\uptau$ for this topology in the remainder and note that by the above $\uppi(K')\subseteq(X,\uptau)$ is closed for every $K'\in\Comp_D$. We claim
$$
(\star)\qquad 
\uppi(\Comp_D) =: \Comp = \bigl\{K\subseteq X\:\big|\:\exists\:K'\in\Comp_D\colon K\subseteq\uppi(K') \text{ is }\uptau\text{-closed}\bigr\},\qquad\phantom{(\star)} 
$$
where `$\subseteq$' follows from what we noted above. For `$\supseteq$' let $K$ and $K'$ be given and consider the map $\varphi\colon(K',\uptau_{K'})\rightarrow(\uppi(K'),\uptau|_{\uppi(K')})$. Then $K'':=\varphi^{-1}(K)\subseteq(K',\uptau_{K'})$ is closed as $\varphi$ is continuous and thus $K''\in\Comp_D$ holds. Since $\varphi$ is surjective, we have $K=\uppi(K'')\in\Comp$.

\smallskip

Straightforward arguments show that $\Comp$ satisfies \ref{DFN-COMP-4}(1)--\ref{DFN-COMP-4}(3) and using the characterization of $\Comp$ that we just established, \ref{DFN-COMP-4}(4) can be checked as well. Using $(\star)$ we see that, for $K\in\Comp$, $\Comp\cap K$ consists precisely of the $\uptau|_K$-closed sets. Referencing \cref{LEM-SHORTCUT}(1), this implies that the $\uptau_K$-closed sets are exactly the $\uptau|_K$-closed ones for all $K\in\Comp$. From this, and since $(K, \uptau|_K)$ is compact, we obtain \ref{DFN-COMP-4}(5) and (6).
Now it is straightforward to check that $\uptau_D$ being final with respect to the subspaces $(K, \uptau_D|_K)$ for $K\in\Comp_D$ implies that $\uptau$ is final with respect to the subspaces $(\uppi(K), \uptau|_{\uppi(K)})$, by virtue of $\uppi$ being a quotient map. It follows that $\topuly{X, \Comp} = (X, \uptau)$ is the coequalizer in $\CGWH$.

\end{proof}

\begin{cor}\label{cor:cmp-bicomp}
    The category of compactological sets is complete and cocomplete.
\end{cor}
\begin{proof} The result follows from Propositions \ref{prop:cmp-equ} and \ref{prop:cmp-prod} using that limits are equalizers of products and colimits are coequalizers of coproducts, see, e.g.\ \cite[Thm 3.4.12]{CatContext}.
\end{proof}

\begin{cor}\label{cor:cmp-po} Let $C=(X_C,\Comp_C)$, $C'=(X_{C'},\Comp_{C'})$, $C''=(X_{C''},\Comp_{C''})\in\Cmp_{\upkappa}$.

\begin{enumerate}
\item Let $f\colon C'\to C$ and $g\colon C''\to C$ be morphisms of compactological sets. The pullback $C'\times_C C''$ of $f$ and $g$ is given by
    \[
         \bigl(X:=\bigl\{ (x, y) \in X_{C'}\times X_{C''} \mid f(x) = g(y) \bigr\},\; \Comp:= X\cap \Comp_{C'\times C''}\bigr).
    \]
    Explicitly, a subset of $X$ belongs to $\Comp$ iff it is a closed subset of some cylinder $K'\times K''$ for $K'\in\Comp_{C'}$ and $K''\in\Comp_{C''}$.\vspace{3pt}

\item Let $f\colon C\to C'$ and $g\colon C\to C''$ be morphisms of compactological sets. The pushout $C'\sqcup_C C''$ is given by
    \[
    \bigl(X:=(X_{C'}\sqcup X_{C''})/E,\; \Comp := \uppi(\Comp_{C'\sqcup C''}) \bigr).
         %\bigl(C'\sqcup C'' \Big/ \overline{\left\{ (f'(x), g'(x)) \mid x\in C \right\}}, \Comp := \pi(\Comp_{C'\sqcup C''}) \bigr),
    \]
    Here, $E$ is the smallest closed equivalence relation containing all $(i_1(f(x)),i_2(g(x)))$ for $x\in X_C$, where $i_1$, $i_2$ denote the inclusions into the disjoint union, and $\uppi$ is the corresponding quotient map.
%Here we denote by $\overline{M}$ for $M\subseteq X\times X$ the closure of $M$ to an equivalence relation on the set $X$, closed as a subset of $C\times C$.
\end{enumerate}
\end{cor}

\begin{proof} Statement (1) follows by using the representation of the pullback $C'\times_C C''$ as the equalizer of $f\circ\uppi_1, g\circ\uppi_2\colon C'\times C''\rightarrow C$. Statement (2) follows dually.
\end{proof}

To put the explicit constructions above into context, we relate them to the more familiar constructions in the category of cgwh-spaces. %Indeed, both limits and colimits in $\Cmp$ can be seen as specializations of their counterparts.

\begin{lem}\label{lem:uly-cont} The forgetful functor $\topuly-\colon \Cmp\rightarrow\CGWH$ is cocontinuous and finitely continuous.
\end{lem}

\begin{proof}Cocontinuity follows by the existence of a right adjoint as given in \cref{CGWH-COMPL-ADJ}. For finite continuity it suffices to show that the functor preserves finite products and equalizers, according to the usual formula for limits, see \cite[Thm 3.4.12]{CatContext}. This was however already shown in \cref{prop:cmp-prod,prop:cmp-equ}.
\end{proof}

We also note the following.

\begin{prop}\label{prop:chaus-incl-cocont} The natural functor $\CHaus_{\upkappa}\into\Cmp_{\upkappa}$, $(K,\uptau)\mapsto(K,\Comp_K)$ with $\Comp_K:=\{(K',\uptau|_{K'})\:|\:K'\subseteq(K,\uptau) \text{ compact}\hspace{1pt}\}$, cf.~\cref{CGWH-COMPL-ADJ}, is finitely cocontinuous.
%The fully faithful functor $\CHaus_\kappa\into\CGWH_\kappa\into\Cmp_\kappa$ preserves finite colimits.
\end{prop}

\begin{proof} By \cite[Thm 3.4.12]{CatContext} it is sufficient to show that the functor preserves finite copro\-ducts and coequalizers. If all objects we start with are $\upkappa$-small, then the same holds for their coproducts/coequalizers taken in $\CHaus$ resp.\ in $\Cmp$. Therefore, we will drop the `$\upkappa$' in the remainder of this proof from our notation.

\smallskip

Firstly we note that finite coproducts in $\CHaus$ are given by the disjoint union, e.g.\ \cite[Ex 2.2.4c]{Borceux}. Given $X,Y\in\CHaus$, let $(X\amalg Y,\Comp_{X\amalg Y})$ be the coproduct taken in $\Cmp$. Using \cref{prop:cmp-prod}(1) we get
$$
\Comp_{X\amalg Y} = \overline{\Comp_{X}\sqcup\Comp_Y}^{\cup_{\text{fin}}} =\bigl\{\fcup_{i=1}^n K_i^{
\scriptscriptstyle(1)}\sqcup K_i^{\scriptscriptstyle(2)}\:\big|\:K_i^{\scriptscriptstyle(1)}\subseteq X,\,K_i^{\scriptscriptstyle(2)}\subseteq Y \text{ compact}, n\in\NN_0\bigr\}   =  \Comp_{X\sqcup Y}.
$$
Moreover, the underlying set of $X\amalg Y$ is the disjoint union of $X$ and $Y$. Thus, $\CHaus\into\Cmp$ preserves finite coproducts.

\smallskip

Let now $f,g\colon X\rightarrow Y$ be morphisms in $\CHaus$ and denote by $Y\times Y$ the product in $\CHaus$, cf.~\cite[Ex 2.1.7g]{Borceux}, which is just the usual product of topological spaces. % It is Tychonoff's Theorem that this is the usual product of topological spaces. 
Denote by $(Y\times Y,\Comp_{Y\times Y})$ the product in $\Cmp$. Using that the topology on $Y$ coming from $\Comp_Y$ is the original topology of $Y$, and using \cref{prop:cmp-prod}(2), we get that the topology $Y\times Y$ given by $\Comp_{Y\times Y}$ is the k-ification of $Y\times Y$. However, by compactness of $Y\times Y$, the product topology is already compactly generated
and hence both topologies are the same. It follows that, as sets, the coequalizer of $f$ and $g$ taken in $\CHaus$ and the coequalizer taken in $\Cmp$ coincide: Indeed, they are both given by exactly the same quotient space as explained in \cref{prop:cmp-equ}. Let $Y/E$ denote the quotient in $\CHaus$ and let $\uppi\colon Y\rightarrow Y/E$ be the quotient map. For the coequalizer's compactology we obtain
\begin{eqnarray*}
\uppi(\Comp_Y)\hspace{-6pt}&=&\hspace{-6pt}\bigl\{\uppi(K)\:\big|\:K\in\Comp_Y\bigl\}\\
&=&\hspace{-6pt}\bigl\{\uppi(K)\:\big|\:K\subseteq Y \text{ compact}\bigl\}\\
& =&\hspace{-6pt}\bigl\{K\subseteq Y/E\:\big|\:K\text{ compact in the quotient topology}\bigl\}\hspace{3pt}=\hspace{1pt}\Comp_{Y/E}
\end{eqnarray*}
since the quotient topology on $Y/E$ coincides with the topology induced by $\uppi(\Comp_Y)$, cf.\ \cref{prop:cmp-equ}(2). \qedhere
\end{proof}

\begin{prop}\label{prop:cmp-gen-compact}
    For any $\upkappa$-compactological set $(X, \Comp)$ we may treat the compactology as a poset-shaped diagram $\Comp\into\CHaus_\upkappa$.
    In this sense
    $$(X,\Comp)\cong\colim_{K\in\Comp} \upiota(K,\uptau_K)$$ 
    holds for all compactological sets $(X, \Comp)$, where $\upiota\colon \CHaus_\upkappa\into\Cmp_\upkappa$ is defined as in \cref{prop:chaus-incl-cocont} and the colimit is taken in $\Cmp$.
    %Every compactological set is the filtered colimit over its diagram of compact spaces of the compactology.
    In particular, the compact spaces $\CHaus_\upkappa\into\Cmp_\upkappa$ generate the category of compactological sets under filtered colimits of diagrams containing only embeddings and this presentation is unique up to refinement of diagrams.
\end{prop}

\begin{proof} 
    We show this by checking the universal property.
    For this let $(Y,\Eomp)\in\Cmp$ be given together with morphisms $f_K\colon K\rightarrow Y$ of compactological sets such that $f_K|_{K'}=f_K'$ holds for all $K,K'\in\Comp$ and $K'\subseteq K$. As $X=\fcup_{K\in\Comp}K$ holds, we get a unique map $f\colon X\rightarrow Y$ with $f|_K=f_K$. By assumption, the $f_K$ are continuous as maps of underlying topological spaces. Furthermore, for $K\in\Comp$ we have $f(K)=f_K(K)\in\Eomp$ again by assumption. Thus $f$ is a morphism of compactological sets. Conversely, it follows by definition that any map of compactological sets induces such a family of functions.

    To see uniqueness up to refinement of this presentation, observe that the sets of the compactology of $\colim \upalpha$ for $\upalpha: \catI\to\CHaus\into \Cmp$ filtered on embeddings are exactly the closed subsets of compacts appearing in $\upalpha$, as included by the colimit cone. It follows that the compactology of $\colim \upalpha$ is a refinement of $\upalpha$ itself and the compactology is therefore a shared refinement for all monomorphic filtered diagrams of compacts describing the same compactological set.
   %We let all compactological spaces be given according to \cref{DFN-COMP-2}. Let $F\colon \Comp\to\CHaus_\upkappa\into\Cmp_\upkappa$ denote the filtered diagram defined by the compactology of an arbitrary $C=(X, \Comp, (\tau_K)_{K\in\Comp})\in\Cmp_\upkappa$. A cone $F\Longrightarrow (Y, \Comp_2, (\tau_{K})_{K\in\Comp_2})$ corresponds to a collection of maps $(f_K\colon K\to X)_{K\in\Comp}$ such that $f_K|_{K\cap K'} = f_{K'}|_{K\cap K'}$ holds for all $K, K'\in\Comp$, $\ran(f_K)$ is always  an element of the compactology $\Comp_2$ and $f_K\colon (K, \tau_K)\to(\ran(f_K), \tau_{\ran(f_K)})$ is continuous.
%
    %It follows from the congruence condition that the maps assemble into a map of sets $f: \bigcup_{K\in\Comp} K = X\to Y$ and by definition of the final topology this map is continuous with respect to the final topologies if all component maps are. This however follows immediately from the assumptions and the continuity of the relevant embeddings $\ran(f_K)\into Y$. Finally, $f(K) = \ran(f_K) \in \Comp_2$ holds for all $K\in\Comp$. Thus any cone under $F$ lifts to a map out of $C$ and this is lift unique by the covering property of the compactology and the fact that any morphism of compactological sets is uniquely identified by its underlying map of sets.
\end{proof}

\begin{lem}\label{lem:cmp-emb-comp}
    Let $K$ be a compact space. The functor $\Cmp(\upiota(K), -): \Cmp\to\Set$ preserves filtered colimits of compact spaces with all transition maps given by embeddings. The analogous result holds under restriction to some uncountable strong limit cardinal $\upkappa$.
\end{lem}

\begin{proof}
    As the proof does not depend on a possible choice of $\upkappa$, we leave it implicit. 
    By definition any map $\upiota(K)\to C$ from a compactum $K$ to a compactological set $C$ factors through some element of the compactology. Hence $\Cmp(\upiota(K), C)$ is given by the monomorphic filtered colimit over $\Cmp(\upiota(K), \upiota(L))$ where $L$ runs through the elements of $\Comp_C$. 
    Referencing \cref{prop:cmp-gen-compact}, we deduce that $\Cmp(\upiota(K), -)$ preserves those filtered colimits induced by compactologies. The uniqueness result of the proposition implies that all filtered colimits of embeddings between compact spaces can be refined to an element in this class and it is a general fact that filtered colimits are invariant under refinement. The claim follows directly from this.
\end{proof}

\begin{prop}\label{prop:ind-cmp}
    The category $\Cmp_\upkappa$ of compactological sets is equivalent to the category $\catstyle{Ind_{emb}}(\CHaus_\upkappa)$ of filtered diagrams of ($\upkappa$-small) compact spaces, with all transition maps given by topological embeddings. The analogous result is true without the restriction to $\upkappa$.
\end{prop}

\begin{proof}
    To prove this, we adapt \cite[Prop 6.3.4]{KashiwaraShapira05}. As the proof does not depend on the restriction, we leave the cardinal $\upkappa$ implicit and take all categories to be appropriately restricted where relevant.
    
    Let $\upalpha: \catI\to\CHaus$, $\upbeta: \catJ\to\CHaus$ be two filtered diagrams of embeddings and denote the functor $\catstyle{Ind_{emb}}\to\Cmp$ sending a formal colimit to some chosen realization in $\Cmp$ by $\underset{\smash[t]{\raisebox{2pt}{$\rightarrow$}}} F$. From \cref{lem:cmp-emb-comp} and the definition of Ind-categories we obtain
    \begin{align*}
        \catstyle{Ind_{emb}}(\CHaus)(\upalpha, \upbeta) :\!&= \lim_{i\in\catI} \colim_{j\in\catJ} \CHaus(\upalpha(i), \upbeta(j)) \\
        &\cong \lim_{i\in\catI} \colim_{j\in\catJ} \Cmp(\upiota(\upalpha(i)), \upiota(\upbeta(j))) \\
        &\cong \Cmp(\colim_\catI \upiota\circ\upalpha, \colim_\catJ \upiota\circ\upbeta) \\
        &\cong \Cmp(\underset{\smash[t]{\raisebox{2pt}{$\rightarrow$}}} F(\upalpha), \underset{\smash[t]{\raisebox{2pt}{$\rightarrow$}}} F(\upbeta)).
    \end{align*}
    Hence $\underset{\smash[t]{\raisebox{2pt}{$\rightarrow$}}} F$ is fully faithful. It follows directly from \cref{prop:cmp-gen-compact} that $F$ is also essentially surjective.
\end{proof}

\section{Quasi-separated condensed sets are compactological}\label{SEC:Equiv}

The beginning of this section contains a streamlined introduction to (quasicompact and quasiseparated) condensed sets. For the convenience of the reader we repeat basic results relevant to our purpose and add several remarks. Our main reference is the first set of lecture notes by Clausen, Scholze \cite{ScholzeCondensedMath}.

\smallskip

Let $\catC\in\{\Set, \Grp, \RMod, \KAlg,\dots\}$ be a category of `algebraic objects'. More formally, let $\catC$ be finitary monadic over $\Set$, see, e.g.\ \cite[Dfns 5.3.1 and 5.5.5]{CatContext}. As previously, we fix an uncountable %sufficiently large
strong limit cardinal $\upkappa$ of cofinality $\uplambda$ and let $\CHaus_\upkappa$ denote the category of compact spaces of cardinality below $\upkappa$. Now we define $\upkappa$-condensed sets as sheaves on $\CHaus_\upkappa$:

\begin{dfn}[\protect{\cite{ScholzeCondensedMath}}]\label{def:cond}
    A \emph{$\upkappa$-condensed set/group/module/algebra/...}, or more generally a \emph{$\upkappa$-condensed $\catC$-object}, is a functor
    \[X\colon\CHaus^{\op}_\upkappa \to \catC,\]
    such that:
    \begin{enumerate}
        \item $X$ maps finite coproducts of $\CHaus_{\upkappa}$ to products in $\catC$.
        \item Any epimorphism $f\colon K\onto L$ induces a bijection %on underlying sets
        \[
                X(f)\colon X(L) \overset{\sim\hspace{1pt}}{\longrightarrow}\bigl\{
                    \upphi\in X(K) \mid X(\uppi_1)(\upphi) = X(\uppi_2)(\upphi)
                \bigr\}
            \]
            for the two projection maps $\begin{tikzcd}
                K\times_L K \ar[r, shift left, "\uppi_1"] \ar[r, shift right, "\uppi_2"'] & K.
            \end{tikzcd}$
    \end{enumerate}
    The category $\Cond_\upkappa(\catC)$ of $\upkappa$-condensed $\catC$-objects is the full subcategory of $\Fun(\CHaus_\upkappa^{\op}, \catC)$ on the $\upkappa$-condensed objects. We write $\Cond_\upkappa:=\Cond_\upkappa(\Set)$ for the category  of $\upkappa$-condensed sets. Given $X\in\Cond_\upkappa$ and denoting by $*$ the singleton, $X(*)$ is the \emph{underlying set}.\hfill\diam
\end{dfn}

\begin{rmk}\label{rem:cond-cond}\begin{myitemize}
    \item[(1)] The restriction to $\upkappa$-small spaces is necessary to ensure that the collection of $\upkappa$-condensed sets is well-defined without passing to stronger axiomatic foundations than classical set theory. By \cite[Prop 2.9 and Rmk 2.10]{ScholzeCondensedMath} for $\upkappa<\upkappa'$, the restriction $\Cond_{\upkappa'}(\catC) \onto \Cond_\upkappa(\catC)$ admits a fully faithful left adjoint $\Cond_\upkappa(\catC)\rightarrow\Cond_{\upkappa'}(\catC)$, which commutes with $\uplambda$-small limits and all colimits.
    %Justified by this, we leave the chosen $\upkappa$ mostly implicit in the following discussion and assume that $\uplambda$ and $\upkappa$ were chosen sufficiently large. 
    To avoid a choice of $\upkappa$ altogether, one may pass to the (large) colimit of the totally ordered diagram over all categories of $\upkappa$-condensed sets. We will denote this category in a few cases of interest by $\Cond$, without an attached cardinal, see \cite[Dfn 2.11]{ScholzeCondensedMath}.%. For further reference, see \cite[Prop 2.9 and Dfn 2.11]{ScholzeCondensedMath}.}
    \vspace{3pt}
        %\item 
            %Using the finitary monadicity of $\catC$ (and the axiom of choice) one may show that the category of $\upkappa$-condensed objects is complete and cocomplete and for $\upkappa<\upkappa'$, the restriction $\Cond_{\upkappa'}(\catC) \onto \Cond_\upkappa(\catC)$ admits a fully faithful left adjoint $\Cond_\upkappa(\catC)\into\Cond_{\upkappa'}(\catC)$, that commutes with $\lambda$-small limits and all colimits.
            %
        \item[(2)] Condition \ref{def:cond}(2) is equivalent to requiring that $X$ maps coequalizers of kernel pairs to equalizers. The used simplification may be derived from the fact that any monadic functor $\catC\to\Set$ reflects isomorphisms. \cref{def:cond}(2) means that the elements of $X(L)$ are precisely those elements of $X(K)$ that are constant on the fibers of $f$, compare with \cref{rem:eval}(1).
        \vspace{3pt}
        \item[(3)] One can equivalently define the category of condensed sets as a category of sheaves on one of multiple different sites \cite[Dfn~1.2, Prop~2.3 and Prop~2.7]{ScholzeCondensedMath}, see \cite{Asgeir}. In fact, the given definition just spells this out for the appropriate site on the category $\CHaus_\kappa$.
        \vspace{3pt}
        %It is equivalent to define the category of condensed sets as the category of sheaves on the site of $\upkappa$-small compact Hausdorff spaces with the coverings given by finite jointly surjective families, which is in turn equivalent to the category of sheaves on the $\upkappa$-restricted pro\'etale site of the point.
        \item[(4)] The requirement that $\catC$ is at least finitary monadic over $\Set$ may be dropped as long as one makes sure that  $\catC$-valued sheaves are sufficiently well-behaved, see for example \cite[Ch 17.4]{KashiwaraShapira05}. As a cautionary example, for which the theory fails to be well-behaved, consider $\catC=\catstyle{Ban}$, the category of Banach spaces.\hfill\diam
        % One may additionally require that the forgetful functor $\catC\to\Set$ is (finitely) continuous and reflects isomorphisms. This ensures that the induced functor $\Cond_\upkappa(\catC)\to\Cond_\upkappa$ is reasonably well behaved. This functor plays a central role in the extension of certain definitions from condensed sets to condensed $\catC$-objects. 
    \end{myitemize}
\end{rmk}

%The monadic adjunction $F \dashv U$ between the free construction $F: \Set\to\catC$ and the forgetful functor $U$
% \[\begin{tikzcd}
%     \catC
%         \ar[r, shift left, bend left, "U"{name=right}] &
%     \Set
%         \ar[l, shift left, bend left, "F"{name=left}]
%         \ar[phantom, from=left, to=right, "\bot"]
% \end{tikzcd}\]
%always lifts to a monad $U'\circ F'$ on $\Cond$. Under the assumption that $U\colon \catC\to\Set$ preserves filtered colimits, its lift commutes with the sheafification. In this situation one may check that $\Cond(\catC) \simeq \Cond^{U'\circ F'}$,\NOTE{Benjamin}{Check that when $U'$ commutes with the sheafification, $U'$ preserves all colimits that $U$ preserves.} where the right hand side denotes the category of $U'\circ F'$-algebras in $\Cond$. The monads for which this holds are precisely the finitary monads. In particular they include the classical algebraic categories, $\catC\in\{\Set, \Grp, \RMod, \KAlg...\}$. Concretely this means that we may interpret condensed abelian groups, condensed modules, etc., as condensed sets equipped with an abelian group structure, respectively a module structure, etc.. We therefore continue by studying the special case of condensed sets. Most definitions, properties and other results can then be carried over to other condensed categories of interest. We return to this at the end of the section.

In the remainder of this article we will restrict ourselves to $\catC=\Set$; the case of `algebraic categories' $\catC$ will be treated in a forthcoming paper.

\begin{dfn} We denote by $\yo: \CHaus_\upkappa \into \Fun(\CHaus_\upkappa^{\op}, \Set)$ the Yoneda embedding, which defines a functor into $\upkappa$-condensed sets. For $K\in\CHaus_\upkappa$ we write $\underline K := \yo(K) = \CHaus_\upkappa(-, K)$ and refer to such condensed sets as \emph{representable}.\hfill\diam
\end{dfn}

Indeed, one can read any topological space $X$ as condensed set $\underline{X}:= \Top(\upiota(-), X)$, see \cite[Ex 1.5]{ScholzeCondensedMath}. We will make use of this only in \cref{thm:qs-hausd}. The term `representable' will always mean represented by a $\upkappa$-small compact space.

%\begin{rmk}\NOTE{Benjamin}{Like this maybe?}
 %   The above functor extends to a functor $\underline{(-)}: \Top\to\Cond_\upkappa$, obtained by restricting the Yoneda embedding along the inclusion $\upiota: \CHaus_\upkappa\into\Top$. We will only reference this concretely in \cref{thm:qs-hausd} and the term \enquote{representable} will always mean representability by a $\upkappa$-small compact space.
%\end{rmk}

% If we read, via the Yoneda Lemma, $\upphi\equiv(\upphi_{K'})_{K'\in\CHaus_{\upkappa}}\colon\CHaus(-,K)\rightarrow X$ as a natural transformation, then it is easy to check that $\upphi_*$ as above coincides with the component $\upphi_{K'}$ for $K'=*$.

\begin{rmk}\label{rem:eval}\begin{myitemize}\item[(1)] By definition,  any element $\upphi\in\underline K(K') = \CHaus_\upkappa(K', K)$ is a map of spaces $K'\to K$ which agrees with the post-composition $\upphi_*\colon\underline K'(*)\to \underline K(*)$ if we identify $K\equiv{}\underline{K}(*)$. Let now $X$ be an arbitrary condensed set and let $\upphi\in X(K)$ be given. By the Yoneda lemma we may read $\upphi\colon\CHaus_\upkappa(-,K)\rightarrow X$ as a natural transformation given by $(\upphi_{L})_{L\in\CHaus_{\upkappa}}$. Specializing $L=*$ and using how the bijection in the Yoneda lemma looks like, we get a map
$$
\upphi_*: K\cong\underline K(*)\to X(*),\;\upphi_*(k)=X(\underline{k})(\upphi)
$$
where we identify $k\in K$ with the map $\underline{k}\colon*\rightarrow K$, $\operatorname{pt}\mapsto k$. Using suggestive notation, we alternatively write $\upphi_*(k)=X(\{k\}\hookrightarrow K)(\upphi)$.

\vspace{3pt}

\item[(2)] The assignment $\upphi\mapsto\upphi_*$ explained in (1) is not necessarily injective outside of representable condensed sets and hence the forgetful functor
$$
U\colon\Cond_{\upkappa}\to\Set,\;X\mapsto X(*),
$$
is in general neither full nor faithful.

\vspace{3pt}

\item[(3)] Let $X$ be a condensed set. Then we may call an element $x\in X(*)$ a `point of $X$'. Using (1) the point $x$ corresponds to the map $x_*\colon\{x\}\cong\CHaus(*,\{x\})\rightarrow X(*)$, which has as range precisely $\{x\}$. %A suggestive, though slightly abusive, notation for points of $X$ is therefore $\{x\}\hookrightarrow X$. 

\end{myitemize}
    % OLD Remark:
    %By definition,  any element $\upphi\in\underline K(K') = \CHaus_\upkappa(K', K)$ is a map of spaces $K'\to K$ which agrees on sets with the post-composition $\upphi_*\colon\underline K'(*)\to \underline K(*)$. This generalizes to arbitrary condensed sets $X$, for which any $\upphi\in X(K)$ defines a map of sets
    %$$
    %\upphi_*: K\cong\underline K(*)\to X(*),\;\upphi_*(k):=\ev_k(\upphi)
    %$$
    %where $\ev_k:=X(\upiota_k: \{k\}\into K)$ is the `evaluation map'. This assignment is consistent with the map $\underline K\to X$ induced by $\upphi$ according to the Yoneda lemma.
    %Notice, that the assignment is not necessarily injective outside of representable condensed sets and hence the forgetful functor
    %$$
    %U\colon\Cond_{\upkappa}\to\Set,\;X\mapsto X(*),
    %$$
    %is in general neither full nor faithful.\hfill\diam{}
\end{rmk}

\begin{dfn}[\protect{\cite[Prop 1.7]{ScholzeCondensedMath}}]\label{def:uly}
    Let $X$ be a $\upkappa$-condensed set. We define the \emph{underlying space} $\topuly X$ of $X$ to be the underlying set $X(*)$ together with the final topology with respect to the collection of maps
    \[
        \bigl\{ \upphi_*: K \to X(*) \mid K\in \CHaus_\upkappa,\;\upphi\in X(K) \bigr\}.
    \]
    This defines a functor $\topuly-\colon\Cond_\upkappa\to\cgTop$ into the category of compactly generated spaces, whose objects we introduced in \cref{CGWH-DFN}(2).\hfill\diam{}
\end{dfn}

Notice that we used the notation $\topuly{-}$ also for the underlying topological space of a \emph{compactological} set in \cref{SEC:Cmp}; we will see later that \cref{def:uly} actually extends our previous definition.

\smallskip

%\begin{lem}\label{lem:repr-concrete}
    %Let $Y\in\Cond_{\upkappa}$ be representable. Then $Y$ is a concrete sheaf, i.e., the forgetful map $\Cond_{\upkappa}(X, Y) \to \Set(U(X), U(Y))$ is injective for any $X\in\Cond_{\upkappa}$.
%\end{lem}

%\begin{proof}
   % If $X$ is representable as well, this  is a direct consequence of the Yoneda lemma and of $\CHaus_{\upkappa}$ being concrete. For arbitrary $X$ it follows from the fact that the collection of all maps with representable domain is epimorphic: This implies that for every two distinct maps $\upphi, \upphi'\colon X\to Y$ there exists some map $\uppsi\colon \yo(K)\to X$ such that $\upphi\circ\uppsi \neq \upphi'\circ\uppsi$. As these composites get mapped faithfully to $\Set$, $U(\upphi)$ and $U(\uppsi)$ must differ already.
%\end{proof}

%We want to define the more general class of condensed sets which are non-degenerate in this sense. We begin with a different class, closely related to the representable condensed sets.

We now want to introduce the two relevant classes of quasicompact and quasiseparated condensed sets. We will see that they are closely related to the appropriate topological namesakes.

\begin{dfn}[\protect{\cite[Dfn 2.5]{ScholzeComplexGeometry}}]
    Let $X$ be any condensed set. A family of morphisms $(\underline{K_i}\to X)_{i\in I}$ is called a \emph{covering} if it is jointly epimorphic:
    \[\fcoprod_{i\in I} \underline{K_i} \onto X.\]
    A condensed set $X$ is called \emph{quasicompact} whenever any covering $(\underline {K_i} \to X)_{i\in I}$ admits a finite subcovering.\hfill\diam{}
\end{dfn}

\begin{lem}[\protect{\cite[Dfn 2.5]{ScholzeComplexGeometry}}]\label{lem:qc-epi}
    A condensed set $X$ is quasicompact if and only if it admits an epimorphism from a representable. \qed
\end{lem}

As a word of caution, we note that while a map of condensed sets is monic iff its components are, the corresponding statement for epimorphisms is in general false. We have the partial converse that any epimorphism $\upeta$ of condensed sets will have epic components $\upeta_E$ between extremally disconnected objects $E$, as these are projective. In particular, $U(\upeta)$ will be surjective. However, for general compact spaces $K$, the component $\upeta_K$ need not be surjective.

\begin{dfn}[\protect{\cite[Dfn 2.6]{ScholzeComplexGeometry}}]
    A condensed set $X$ is \emph{quasiseparated} if pullbacks along the diagonal inclusion $X\into X\times X$ preserve quasicompactness. Reformulating this, $X$ is quasiseparated iff for arbitrary $K\in\CHaus_{\upkappa}$ any pullback of the following form is quasicompact:
    \[\begin{tikzcd}
        \underline K\times_X \underline K \ar[r, dashed] \ar[d, dashed] \ar[dr, phantom, "\lrcorner" very near start] &
        \underline K \ar[d, "f"] \\
        \underline K \ar[r, swap, "g"] &
        X.
    \end{tikzcd}\]
    We write $\qsCond_{\upkappa}$ for the full subcategory formed by the quasiseparated condensed sets.\hfill\diam{}
\end{dfn}

\begin{lem}[\protect{\cite[Prop 2.8]{ScholzeComplexGeometry}}]\label{lem:qcqs-chaus}
    The subcategory $\catstyle{qcqsCond}_{\upkappa}$ of quasicompact quasiseparated condensed sets is equivalent to the subcategory of the representable condensed sets and hence to the category $\CHaus_{\upkappa}$ of compact spaces. \qed
\end{lem}

\begin{lem}[\protect{\cite[Lem 4.14]{ScholzeAnalyticGeometry}}]\label{lem:qs-refl} The inclusion $\qsCond_{\upkappa}\into\Cond_{\upkappa}$ admits a finite-product-preserving left adjoint, exhibiting the former as a reflective subcategory. \qed
\end{lem}

By design, quasiseparated condensed sets are the appropriate condensed analogy to weak Hausdorff spaces. Concretely, we have the following result.

\begin{thm}[\protect{\cite[Thm 2.16]{ScholzeCondensedMath}}]\label{thm:qs-hausd}
    A ($\upkappa$-)compactly generated topological space $T$ is weak Hausdorff if and only if the ($\upkappa$-)condensed set $\underline T$ is quasiseparated. Conversely, the underlying topological space $\topuly X$ of any quasiseparated ($\upkappa$-)condensed set $X$ is a ($\upkappa$-)cgwh-space. \qed
\end{thm}

Complementing this with the following result, which was mentioned in \cite[p.~9]{ScholzeAnalyticGeometry}, we obtain that quasiseparated condensed sets are both `weak Hausdorff' and such that assigning $\upphi\in X(K)$ to $\upphi_*: K\to X(*)$ as in \cref{rem:eval} is an injective map. This latter property can be understood as a non-degeneracy condition, saying that a quasiseparated condensed set is `just' a set with some structure; see also the explanations in \cite[p.~5803--5805]{BaezHoff} on `concrete sheaves'.

\begin{lem}\label{lem:qs-ff}
    Let $Y\in\qsCond_{\upkappa}$ be a quasiseparated condensed set. Then the forgetful map $U\colon \Cond_{\upkappa}(X, Y) \to \Set(U(X), U(Y))$ is injective for any $X\in\Cond_{\upkappa}$. Hence, any such $Y$ is a concrete sheaf. In particular this holds for all representable $Y\in\Cond_{\upkappa}$.
\end{lem}

\begin{proof} 
    First notice that if $Y$ is representable, then it is quasiseparated by \cref{thm:qs-hausd}. Let now $\upeta, \upxi\colon X\to Y$ be two morphisms in $\Cond_{\upkappa}$, such that $U(\upeta)=U(\upxi)$. Choose some $\upphi \in X(K)$ for arbitrary $K\in \CHaus_{\upkappa}$. We will show that $\upeta_K(\upphi)=\upxi_K(\upphi)$, regardless of $\upphi$, and hence $\upeta = \upxi$. By the Yoneda lemma $\upeta_K(\upphi)\equiv \upeta\circ\upphi\colon \underline K \to Y$ and $\upxi_K(\upphi)\equiv\upxi\circ\upphi\colon \underline K\to Y$ are morphisms of condensed sets with representable domain, identified under $U$. As $Y$ is quasiseparated, the pullback $\underline K\times_{Y\times Y} Y =: Z$ of the following diagram is quasicompact:
    \[\begin{tikzcd}[ampersand replacement=\&, column sep=large]
        Z   \ar[r, dashed, "\uppi_1"]
            \ar[d, dashed, "\uppi_2"'] 
            \ar[dr, phantom, "\lrcorner" very near start] \&
        Y \ar[d, "\Delta"] \\
        \underline K \ar[r, "{(\upeta\upphi, \upxi\upphi})"'] \&
        Y\times Y.
    \end{tikzcd}\]
    According to \cref{lem:qc-epi}, there exists an epimorphism $\underline L \onto Z$. 
    Using pointwise computation, we have that $Z(*)$ is the equalizer of $U(\upeta\upphi)$ and $U(\upxi\upphi)$ and hence isomorphic to $K(\ast)$. Moreover, $\topuly-$ is cocontinuous and hence preserves epimorphisms, which implies that $\uppsi\colon L\onto \topuly Z \onto K$ is an epimorphism of topological spaces, whose image in $\Cond_\upkappa$ equalizes $\upeta\circ\upphi$ and $\upxi\circ\upphi$. From this, we may use \cref{def:cond}(2), which implies that $X(\uppsi)$ and $Y(\uppsi)$ are injective. Using the Yoneda lemma and naturality of $\upeta$, $\upxi$, $\upphi$ and $\uppsi$, we obtain the following diagram:
    \[\begin{tikzcd}[row sep=large]
        & \id_L \in \underline L(L)
            \ar[d, mapsto, "\uppsi_L"'] \\
            % \ar[r, equals] &
        % \underline L(L) \ni \id_L
            % \ar[d, mapsto, "\uppsi_L"] & \\
        \id_K\in\underline K(K)
            \ar[r, mapsto, "\underline K(\uppsi)"]
            \ar[d, mapsto, "\upphi_K"'] &
        \uppsi\in\underline K(L)
            \ar[d, "\upphi_L"'] &
            % \ar[r, equals] &
        % \underline K(L) \ni \uppsi
            % \ar[d, mapsto, "\upphi_L"] &
        \underline K(K) \ni \id_K
            \ar[l, mapsto, "\underline K(\uppsi)"']
            \ar[d, mapsto, "\upphi_K"] \\
        \upphi\in X(K)
            \ar[r, hook, "X(\uppsi)"]
            \ar[d, mapsto, "\upeta_K"'] &
        X(L)
            % \ar[r, equals]
            \ar[d, shift right, "\upeta_L"'] % &
        % X(L)
            \ar[d, shift left, "\upxi_L"] &
        X(K) \ni \upphi
            \ar[l, hook', "X(\uppsi)"']
            \ar[d, mapsto, "\upxi_K"] \\
        \upeta_K(\upphi) \in Y(K)
            \ar[r, hook, "Y(\uppsi)"] &
        % Y(L)
            % \ar[r, equals] &
        Y(L) &
        Y(K) \ni \upxi_K(\upphi).
            \ar[l, hook', "Y(\uppsi)"']
    \end{tikzcd}\]
    By chasing $\upphi$ through this diagram we conclude
    \begin{align*}
        Y(\uppsi)(\upeta_K(\upphi)) &= \upeta_L(X(\uppsi)(\upphi)) = \upeta_L(\upphi_L(\uppsi)) = (\upeta\circ\upphi\circ\uppsi)_L(\id_L)\\
        &= (\upxi\circ\upphi\circ\uppsi)_L(\id_L) = \upxi_L(\upphi_L(\uppsi)) = \upxi_L(X(\uppsi)(\upphi)) = Y(\uppsi)(\xi_L(\upphi))
    \end{align*}
    and by injectivity of $Y(\uppsi)$ we get $\upeta_K(\upphi) = \upxi_K(\upphi)$ as desired.
\end{proof}

\begin{lem}\label{lem:rest-yon} Let $\upiota\colon\CHaus_{\upkappa}\hookrightarrow\Cmp_{\upkappa}$, $(K,\uptau)\mapsto(K,C_K)$, be the embedding from \cref{CGWH-COMPL-ADJ}. The restricted Yoneda embedding
    \begin{align*}
        \yo: \Cmp_\upkappa &\to \Fun(\CHaus^{\op}_\upkappa, \Set) \\
            C &\mapsto \Cmp_{\upkappa}(\upiota(-), C)
    \end{align*}
    defines a fully faithful functor from the category of $\upkappa$-compactological sets into the category of $\upkappa$-condensed sets.
\end{lem}

\begin{proof}
    Both properties follow since every compactological set is a colimit of compact spaces according to \cref{prop:cmp-gen-compact}, as this implies that the components of a natural transformation $\yo(C_1)\to\yo(C_2)$ are fully determined by and uniquely extend from their restrictions to the subcategory described by $\upiota$.
    %
    %Fully faithfulness of the restricted Yoneda embedding follows from the fact that $\upkappa$-small compact spaces generate the category of $\upkappa$-compactological sets, according to \cref{prop:cmp-gen-compact}.
    To see that the functor maps to the subcategory of condensed sets, observe that \ref{def:cond}(1) and, using \cref{rem:cond-cond}(2), \ref{def:cond}(2) both hold if $\yo(C)\colon \CHaus^{\op}_\upkappa\rightarrow\Set$ preserves finite limits. We know that the inclusion $\upiota\colon\CHaus_\upkappa\to\Cmp_\upkappa$ preserves finite colimits due to \cref{prop:chaus-incl-cocont}. Since the Hom-functor $\Cmp_{\upkappa}(-,C)\colon\Cmp_{\upkappa}^{\op}\rightarrow\Set$ preserves limits, the conclusion follows. %it follows that all finite limits of $\CHaus^{\op}_\upkappa$ get mapped to finite limits of $\Cond_\upkappa$.
\end{proof}

\begin{lem}\label{lem:uly-cmp}
Let $X\in\qsCond_\upkappa$ be a quasiseparated condensed set. Then 
\[
\Comp_{X} := \{ \ran(\upphi_*)\subseteq X(*) \mid \upphi \in X(K),\, K\in\CHaus_\upkappa \}
\]
defines a compactology on $X(*)$ and $\cmpuly-\colon \qsCond_\upkappa\to\Cmp_\upkappa$, $X\mapsto(X(*),\Comp_X)$, induces a functor assigning to $X$ its \emph{underlying compactological set}.
\end{lem}

\begin{proof} We first show that $\Comp_{X}$ is indeed a compactology by verifying \ref{DFN-COMP-4}(1)--\ref{DFN-COMP-4}(6).%; in doing so we will identify $X(K)\equiv\Nat(\CHaus_{\upkappa}(-,K),X)=\Cond_{\upkappa}(\underline{K},X)$ according to Yoneda's lemma without mentioning this explicitly every time.

\smallskip

\ref{DFN-COMP-4}(1): Since $\varnothing=\coprod_{\varnothing}$ holds in $\CHaus_{\upkappa}$,  \cref{def:cond}(1) implies that there is $\upphi\in X(\varnothing)=\prod_{\varnothing}=*$. Then $\upphi_*\colon\varnothing\cong\CHaus_{\upkappa}(*,\varnothing)\rightarrow X(*)$ has an empty range.

\smallskip

\ref{DFN-COMP-4}(2): Follows from \cref{rem:eval}(3).

\smallskip

\ref{DFN-COMP-4}(3): For $\ran(\upphi_*),\ran(\uppsi_*)\in\Comp_{X}$ we have $(\upphi,\uppsi)\in X(K)\times X(L)\cong X(K\sqcup L)$ due to \ref{def:cond}(1). Let $k\in K$ be given and consider the maps $\underline{k}\colon *\rightarrow K$, $\text{pt}\mapsto k$, $i_1\colon K\rightarrow K\sqcup L$, $k\mapsto(k,1)$ and $\underline{(k,1)}\colon *\rightarrow K\sqcup L$,  $\text{pt}\mapsto (k,1)$. Using that $X(i_1)=\uppi_1$ due to \ref{def:cond}(1), we obtain
$$
(\upphi,\uppsi)_*\bigl((k,1)\bigr) = X(\hspace{0.7pt}\underline{k}\hspace{0.7pt})X(i_1)(\upphi,\uppsi)= X(\hspace{0.7pt}\underline{k}\hspace{0.7pt})\hspace{1pt}\uppi_1(\upphi,\uppsi) = X(\hspace{0.7pt}\underline{k}\hspace{0.7pt})(\upphi)
$$
and we may conclude that $\ran((\upphi,\uppsi)_*)=\ran(\upphi_*)\cup\ran(\uppsi_*)$.

\smallskip

\ref{DFN-COMP-4}(4): Employing \cref{def:uly} we observe that the elements of $\Comp_X$ are necessarily compact subsets of $\topuly{X}$. In order to get that $\Comp_{X}$ is stable under intersections over non-empty index sets, it is thus sufficient to show 
$$
(\circ)\;\;\;\forall\:K\in\Comp_{X},\,L\subseteq K\text{ with }L\subseteq\topuly{X}\text{ closed}\colon L\in\Comp_{X}.
$$
We first observe the following: Given $K\in\Comp_{X}$, there exists some $K'\in\CHaus_{\upkappa}$ and $\upphi\in X(K')$ with $K=\ran(\upphi_*)$ by definition. However, using \cref{def:uly} and \cref{lem:qs-ff}, we see that $\upphi_*\colon K'\rightarrow K\subseteq X(*)$ is a surjection of compacts and we find $\uppsi\in X(K)$ such that $X(\upphi_*)(\uppsi)=\upphi$ according to \cref{def:cond}(2). As in the previous paragraph, naturality shows $K=\ran(\uppsi_*)$ which means that we may w.l.o.g.\ assume $K=K'$ and that the map $\uppsi_*\colon K\rightarrow X(*)$ is the inclusion map.

\smallskip

To establish $(\circ)$  let now $K=\ran(\upphi_*)$ with $\upphi\in X(K)$ and $i\colon L\hookrightarrow K$ be the inclusion map of a closed subspace. We put $\uppsi:=X(i)(\upphi)\in X(L)$ and, again, by naturality, $L=\ran(\uppsi_*)$ holds.

\smallskip

\ref{DFN-COMP-4}(5) and \ref{DFN-COMP-4}(6): To apply \cref{LEM-SHORTCUT}(4), let us show that on $K\in\Comp_{X}$ the topology $\uptau_K$ as defined in \cref{DFN-COMP-4} and the topology induced by $\topuly{X}$ coincide. Let first $F\subseteq(K,\uptau_K)$ be closed. Then by \cref{LEM-SHORTCUT}(1) we have $F\in\Comp_{X}$, which means that $F\subseteq\topuly{X}$ is closed by what we noted in the paragraph after $(\circ)$. For the other direction assume that $F\subseteq K$ is closed w.r.t.\ the topology induced by  $\topuly{X}$. Then $F\subseteq\topuly{X}$ is closed and by $(\circ)$ we obtain $F\in\Comp_{X}$.

\smallskip

It remains to check that the assignment $X\mapsto(X(*),\Comp_{X})$ is a functor. For this let $\upeta\in\Cond_{\upkappa}(X,Y)$ and let $K=\ran(\upphi_*)\in\Comp_X$ with $\upphi\in X(K)$ be given. As $\upeta$ is a natural transformation we have for $k\in K$ the diagram
\[
\begin{tikzcd}
X(K)\ar[r, "\upeta_K"]\ar[d,  "X(\underline{k})"'] & \underline Y(K) \ar[d, "Y(\underline{k})"] \\
X(*) \ar[r, "\upeta_\ast"'] & Y(*).
\end{tikzcd}
\]
Putting $\uppsi:=\upeta_K(\upphi)\in Y(K)$ the diagram yields $\upeta_*(\ran(\upphi_*))=\ran(\uppsi_*)$, and thus $\upeta$ maps $\Comp_X$ into $\Comp_Y$. By \cref{def:uly} we know that $\upeta\colon\topuly{X}\rightarrow\topuly{Y}$ is continuous and for $K\in\Comp_X$ the topology $\uptau_{\hspace{0.5pt}\upeta(K)}$ coincides with the topology induced by $\topuly{Y}$. Thus $\upeta$ is a morphism in $\Cmp_\upkappa$.
\end{proof}

Observe that $\Comp_X$ does in general not consist of all compact subsets of $\topuly{X}$. On the other hand, the paragraph after $(\circ)$ in the above proof shows that for $X\in\qsCond_{\upkappa}$ the topology induced by $\Comp_X$ on $X(*)$ equals the topology of $\topuly{X}$. The next result establishes in detail the equivalence that was mentioned without much elaboration in \cite[p. 15]{ScholzeComplexGeometry}.

\thmcmpcondqs
%\begin{thm}\label{thm:cmp-condqs}
   % The restricted Yoneda embedding and the underlying compactological set induce an equivalence of categories
   % \[
     %   \yo: \Cmp_\upkappa \overset\sim\adjarrow \qsCond_\upkappa :\cmpuly-
   % \]
    %between $\upkappa$-compactological and quasiseparated $\upkappa$-condensed sets.
%\end{thm}

\begin{proof} 
    For the proof we will drop $\upkappa$ and also $\upiota\colon\CHaus\rightarrow\Cmp$ from notation and understand the involved categories as appropriately restricted. By \cref{lem:rest-yon} the restricted Yoneda embedding $\yo\colon\Cmp\to\Cond$ is fully faithful. It suffices thus to show that for any $X\in\qsCond$ there are isomorphisms $X(K) \cong \yo(\cmpuly X)(K)$ that are natural in $K$; this will imply then that $\yo$ is essentially surjective. It will follow from our construction that these are also natural in $X$, thereby exhibiting $\cmpuly-$ as the desired quasi-inverse.
    %with $\cmpuly-$ as its quasi-inverse.
    
    \smallskip
    
    We note first that $\cmpuly{\underline K} \cong K$ holds naturally for any $K\in\CHaus$ by explicit computation and show next that for all $X\in\qsCond, K\in\CHaus$ we have a natural bijection given by $\cmpuly-$:
    \[
    (\bullet)\;\;\;\;\;\Cond(\underline K, X) \longrightarrow \Cmp(K, \cmpuly X).
    \]
    Indeed, injectivity follows from \cref{lem:qs-ff}. To see surjectivity let $f\in\Cmp(K,\cmpuly{X})$ be given. Then we have $\ran f\in\Comp_X$ as defined in \cref{lem:uly-cmp}. Consequently, we may read $\ran f\in\CHaus$ and further there exist $K'\in\CHaus$, $\uppsi\in X(K')$ with $\ran f = \ran\uppsi_*$. From this we get natural maps
    $$
    \overline{\uppsi}\colon K'\onto \ran f \;\;\text{ and }\;\; i\colon\ran f\hookrightarrow \cmpuly{X},
    $$
    where $\overline{\uppsi}$ is an epimorphism between compact spaces and $i$ is a morphism of compactological sets. We claim that $\uppsi\in\ran X(\overline{\uppsi})$ holds. For this denote by  $\uppi_1, \uppi_2\colon K'\times_{\ran f} K' \to K'$ the projections. For $j=1,2$ the Yoneda Lemma yields
    \[
    \begin{tikzcd}
    X(K')\ar[r, "X(\uppi_j)"]\ar[d,  "\sim"'] & X(K'\times_{\ran f}K')  \ar[d, "\sim"] \\
    \Cond(\underline{K}',X) \ar[r, "\uppi_j^{\hspace{1pt}*}"'] & \Cond(\underline{K'\times_{\ran f}K'},X)
    \end{tikzcd}
    \]
    and thus $X(\uppi_1)\uppsi=X(\uppi_2)\uppsi$ holds in $X(K'\times_{\ran f}K')$ if and only if $\uppi_1^{\hspace{1.5pt}*}\hspace{0.5pt}\uppsi=\uppi_2^{\hspace{1.5pt}*}\hspace{0.5pt}\uppsi$ holds in $\Cond(\underline{K'\times_{\ran f}K'},X)$. By \cref{lem:qs-ff} the latter is equivalent to $(\uppi_1^{\hspace{1.5pt}*}\hspace{0.5pt}\uppsi)_*=(\uppi_2^{\hspace{1.5pt}*}\hspace{0.5pt}\uppsi)_*$ in $\Set(K'\times_{\ran f}K',X(*)),$ which holds by definition of the fiber product using that we have $(\uppi_j^{\hspace{1.5pt}\hspace{0.5pt}*}\hspace{0.5pt}\uppsi)_*=\uppsi_*\circ\uppi_j$ for $j=1,2$. This establishes the claim. Invoking \cref{def:cond}(2) there exists $\mathfrak i\in X(\ran f)$ such that
    $$
    X(\overline{\uppsi})(\mathfrak i)=\uppsi.
    $$
    We read $\mathfrak i\in\Cond(\underline{\ran f},X)$ and observe that $\cmpuly{\mathfrak i}\colon \cmpuly{\underline{\ran f}}\cong\ran f\rightarrow \cmpuly{X}$ coincides with the inclusion $i$: Indeed, by our selection of $\mathfrak i$ above, and by the very definitions of $\overline{\uppsi}$ and $i$, we obtain ${\mathfrak i}_*\circ\overline{\uppsi}=\uppsi_*=i\circ\overline{\uppsi}$; since $\overline{\uppsi}$ is surjective it then follows ${\mathfrak i}_*=i$ as maps of sets and thus $\cmpuly{\mathfrak i}=i$. Now we consider
    $$
    X(f|^{\ran f})\colon X(\ran f)\rightarrow X(K)
    $$
    and define $\tilde{f}:=X(f|^{\ran f})(\mathfrak i)\in X(K)$ which we may read as a morphism of sheaves $\tilde{f}\colon\underline{K}\rightarrow X$. It follows $(f|^{\ran f})^*\hspace{1pt}{\mathfrak i}=\tilde{f}$ from the definition of $\tilde{f}$ and thus
    $$
    (\tilde{f}\hspace{1.25pt})_*=((f|^{\ran f})^*\hspace{1pt}{\mathfrak i}\hspace{0.5pt})_*={\mathfrak i}_*\circ f|^{\ran f} = i\circ f =  f,
    $$
    which means that $\cmpuly{\tilde{f}\hspace{1pt}}=f$ holds as morphisms in $\Cmp$ and establishes that the map in $(\bullet)$ is surjective. That the bijection is natural in $X$ and $K$ holds because, at the level of sets, we are only performing evaluations in $*$.
    
    \smallskip
    
    Having the natural bijection $(\bullet)$, Yoneda's Lemma yields the desired natural isomorphism $X(K) \cong \Cond(\underline K, X) \cong \Cmp(K, \cmpuly X) = \yo(\cmpuly X)(K)$, which finishes the proof in the restricted case. 

    \smallskip

    For the unrestricted case, observe that 
    %by \cite[Prop 2.15]{ScholzeCondensedMath} 
    the Yoneda embedding of any $\CHaus_\upkappa$ into $\Cond$ factors through all $\Cond_{\upkappa'}$ with $\upkappa'$ greater or equal to $\upkappa$. Because any $C\in\Cmp_\upkappa$ is a colimit over $\upkappa$-small compacts and the left adjoints $\Cond_\upkappa\to\Cond_{\upkappa'}$ are cocontinuous, the same is true for the embeddings $\Cmp_\upkappa\into\Cond$. Finally, it follows by preservation (and reflection) of finite limits and epimorphisms by any and all of the inclusions $\Cond_\upkappa\into\Cond_{\upkappa'}$ that quasiseparatedness of a condensed set $X\in\Cond_\upkappa$ embedded into $\Cond_{\upkappa'}$ stabilizes for sufficiently large $\upkappa'$. Thus, choosing such a $\upkappa'$ for every $\upkappa$, we may construct a cone of fully faithful functors $\Cmp_\upkappa\into\Cmp_{\upkappa'}\into \Cond_{\upkappa'} \to \Cond$ under the directed diagram of categories $\Cmp_\upkappa$. The functor $\Cmp\to\Cond$ corresponding to this cone will still be fully faithful and its essential image will, by construction, correspond to those condensed sets that are quasiseparated for all cardinals above some $\upkappa'$. These are exactly the condensed sets quasiseparated in $\Cond$.
\end{proof}

Using \cref{lem:qs-refl} and identifying ($\upkappa$-)compactological with quasiseparated ($\upkappa$-)condensed sets via  \cref{thm:cmp-condqs} we get the following.

\begin{cor}\label{cor:cmp-refl} Compactological sets form a reflective subcategory of condensed sets. % and we henceforth implicitly identify the subcategory of quasiseparated condensed sets with the category of compactological sets. 
We obtain a cocontinuous and finite-product-preserving \emph{compactologification functor} $\cmpuly-\colon\Cond_\upkappa\to\qsCond_\upkappa\cong\Cmp_\upkappa$ as the left adjoint of $\Cmp_{\upkappa}\hookrightarrow\Cond_{\upkappa}$. This adjunction lifts to the directed colimits $\Cmp$ and $\Cond$. \qed
\end{cor}

We conclude this chapter with the following result, which has also been noted by Scholze and Clausen in \cite[Prop 1.2(4)]{ScholzeAnalyticGeometry}.

\begin{cor}\label{cor:ind-condqs}
    The category of quasiseparated $\upkappa$-condensed sets is equivalent to the category $\catstyle{Ind_{emb}}(\CHaus_\upkappa)$ of filtered diagrams of $\upkappa$-small compact spaces, with all transition maps given by topological embeddings. The functor $\underset{\smash[t]{\raisebox{2pt}{$\rightarrow$}}} F\colon \catstyle{Ind_{emb}}(\CHaus_\upkappa)\to\qsCond_\upkappa$ exhibiting this equivalence sends a formal colimit to its realization in $\Cond_\upkappa$, where a compact space $K$ is identified with $\underline K \cong \yo(\upiota(K))$. In particular, the inclusion $\qsCond_\upkappa \into \Cond_\upkappa$ preserves this class of colimits.
\end{cor}

\begin{proof}
    The equivalence of categories follows directly from \cref{prop:ind-cmp,thm:cmp-condqs}. Unraveling the definitions it follows that the functor $\smash{\underset{\smash{\raisebox{2pt}{$\rightarrow$}}} F}\colon \catstyle{Ind_{emb}}(\CHaus_\upkappa)\to\qsCond_\upkappa$ is given by $\upalpha \mapsto \Cmp_\upkappa(\upiota(-), \colim \upalpha)$ for $\upalpha\in\catstyle{Ind_{emb}}(\CHaus_\upkappa)$. Now, according to \cref{lem:cmp-emb-comp}, $\Cmp_\upkappa(\upiota(K), -)$ commutes with filtered colimits of compact spaces and embeddings for all $K\in\CHaus_\upkappa$. Finally, commutativity of filtered colimits with finite limits in $\Set$ implies that filtered colimits of condensed sets are computed pointwise, because the sheaf condition is preserved. We obtain
    \begin{align*}
        (\colim \yo\circ\upiota\circ\upalpha)(K) &\cong \colim ((\yo\circ\upiota\circ\upalpha)(K)) \\
        &\cong \colim \Cmp_\upkappa(\upiota(K), \upiota\circ\upalpha) \\
        &\cong \Cmp_\upkappa(\upiota(K), \colim \upiota\circ\upalpha) \\
        &\cong \underset{\smash[t]{\raisebox{2pt}{$\rightarrow$}}} F(\upalpha)(K),
    \end{align*}
    which proves the claim.
\end{proof}

\section{The category of compactological sets}\label{SEC:CmpCat}

In order to highlight that compactological sets are not just a classical description of certain condensed objects but form an interesting category of spaces in their own right, we employ the usual machinery of category theory to derive multiple novel results on compactological sets. In particular the categories $\Cmp_\upkappa$ and $\Cmp$ are cartesian closed and both regular and coregular, inheriting the distinguishing properties of $\CGWH$ while being exceptionally well described as formal and non-degenerate gluings of compact spaces according to \cref{prop:ind-cmp}. %Finally, the characterization already given in \cref{DFN-COMP-4} exhibits compactological sets as comparably \enquote{primitive} to other notions of \enquote{nice} topological spaces.

\smallskip

We begin by subsuming \cref{cor:cmp-bicomp} with the significantly stronger property of totality. Due to the set-theoretic considerations noted in \cref{rem:cond-cond} this result has to be treated with care, as neither $\Cmp$ nor $\Cmp_\upkappa$ is (essentially) small and therefore $\PSh{\Cmp_\upkappa}$ is no longer a (locally small) category. However, making use of the fact that $\Cmp_\upkappa$ is essentially determined by the small subcategory $\CHaus_\upkappa$, the argument below shows that there exists a class-assignment $\PSh{\Cmp_\upkappa}\to\Cmp_\upkappa$, which is left adjoint to $\yo$ as far as that notion is defined. This is equivalent to the definition of totality based on the existence of weighted colimits, as the assignment we construct exactly computes these. In particular the proof can be carried out unrestricted when assuming the existence of a larger set-theoretic universe rendering $\PSh{\Cmp_\upkappa}$ a well-defined category. Under such an assumption the result extends to $\Cmp$.

\begin{prop}\label{prop:cmp-total}
    The category of compactological sets is total, that is, the Yoneda embedding $\yo\colon\Cmp_\upkappa\into\PSh{\Cmp_\upkappa}$ admits a (class-sized) left adjoint. It follows that the category is complete, cocomplete and admits all large limits and colimits that a locally small category can admit.
\end{prop}

For a precise treatment of totality, including clear characterizations of which large limits and colimits exist, we refer to \cite[Sec.~5]{Kelly86Totality}.
Abstractly the above proposition follows as $\Cmp_\upkappa$ is a reflective subcategory of the topos $\Cond_\upkappa$ according to \cref{cor:cmp-refl}. As any (Grothendieck) topos is total, it is a reflective subcategory of a total category and therefore itself total. In the following we provide an explicit construction for the left adjoint. Because this construction only depends on the restriction of any presheaf $\Cmp_\upkappa^{\op}\to\Set$ along $\CHaus_\upkappa\into\Cmp$, this proof provides an explicit description of the left adjoint of the inclusion $\Cmp_\upkappa\cong\qsCond_\upkappa\into\Cond_\upkappa$.

\begin{proof}
    We will construct a compactological set associated to a presheaf in \textcircled{1} and show that this construction is canonically functorial in \textcircled{2}. We will then exhibit this class-sized functor as `left adjoint' to the Yoneda embedding in \textcircled{3} to \textcircled{6}.

\smallskip

    \textcircled{1} Let $F: \Cmp^{\op}_\upkappa\to\Set$ be a presheaf over compactological sets. By the Yoneda lemma the elements of the set $F(C)$ correspond bijectively to maps $\yo(C)\to F$, where $C\in\Cmp_\upkappa$ is some $\upkappa$-compactological set. Evaluation at the point $\ast\in\Cmp_\upkappa$ induces a map of sets $\yo(C)(\ast)\cong C\to F(\ast)$. We aim to equip the set $F(\ast)$ with the final topology with respect to all such maps. Making use of \cref{prop:cmp-gen-compact}, this final topology is defined by the small set of maps out of $\upkappa$-compact spaces, rendering $F(\ast)$ a compactly generated space.

\smallskip
    
    Now, according to \cite[Prop 2.22]{CGWH}, there is a reflective inclusion $\CGWH\into\cgTop$, with reflector $h\colon\cgTop\onto\CGWH$. That this restricts to $\upkappa$-compactly generated and $\upkappa$-cgwh spaces follows by definition. Importantly, the unit of this adjunction consists of regular epimorphisms $T\onto hT$ for $T\in\cgTop$. By definition, $hF(\ast)$ is $\upkappa$-compactly generated by all maps $K\into C\to F(\ast) \onto hF(\ast)$ defined through some element $\upphi\in F(C)$ and some element of the compactology $K\in\Comp_C$, which is always $\upkappa$-small. Because $hF(\ast)$ is weak Hausdorff, we may use that the different versions of compact generation are equivalent and the topology on $hF(\ast)$ is generated by all sets of the form 
    \[
        \ran(K\into C\to F(\ast)\onto hF(\ast)).
    \]
    By naturality, this system of sets is closed under inclusions of compacts, as the restriction $F(f)(\upphi)$ of some $\upphi\in F(C)$ along $f\colon K\into C$ for $K\in\Comp_C$ always induces a map $K\to F(\ast)$, which is the restriction of $\upphi: C\to F(\ast)$. It is also covering by the fact that $F(\ast)$ bijects with maps $\ast\into F(\ast)$. Closing it under finite unions we obtain a $\upkappa$-compactology on $hF(\ast)$ by the same argumentation as in \cref{lem:uly-cmp}.

\smallskip
    
    \textcircled{2} From (class-sized) functoriality of the evaluation $F\mapsto F(\ast)$ and the universal characterization of the final topology, we obtain that any natural transformation $\upeta\colon F\to G$ induces a map $F(\ast)\to G(\ast)$, which will be continuous with respect to the final topology on the two spaces. Now functoriality of the weak-Hausdorffification $F(\ast)\to hF(\ast)$ ensures that our construction is functorial into the category of cgwh-spaces.
    Finally, for any natural transformation $\upeta\colon F\to G$, it follows directly from naturality that the induced map $h\upeta_\ast\colon hF(\ast) \to hG(\ast)$ must map $\ran(\upphi)$ to $\ran(\upeta_C(\upphi))$ for $\upphi\in F(C)$, $C\in\Cmp_\upkappa$ and thus the functorially associated continuous maps of topological spaces lift to morphisms of compactological sets.

\smallskip
    
    \textcircled{3} It remains to show that this construction is `left adjoint' to the Yoneda embedding. We will construct maps between $\Cmp_\upkappa(hF(\ast), C)$ and $\PSh{\Cmp_\upkappa}(F, \yo(C))$, natural in both $F$ and $C$, which we will show to be inverse to each other. 
    For this, observe first that any morphism of presheaves $\upeta\colon F\to \yo(C)$ is uniquely determined by $\upeta_\ast\colon F(\ast)\to C$, i.e., by its behavior on sets: Indeed, any two distinct morphisms $\upeta,\upepsilon\colon F\to \yo(C)$ admit some $\upphi\colon \yo(C')\to F$ such that $\upeta\circ\upphi\neq \upepsilon\circ\upphi$. The claim now follows because any map $\yo(C')\to\yo(C)$ is determined by its behavior on sets according to the corresponding property of compactological sets.
    %For this, observe first that any morphism of presheaves $\upeta\colon F\to \yo(C)$ is uniquely determined by its behaviour on sets $\upeta_\ast\colon F(\ast)\to C$, owing to any two distinct morphisms $\upeta,\upepsilon\colon F\to \yo(C)$ admitting some $\upphi\colon \yo(C')\to F$ such that $\upeta\circ\upphi\neq \upepsilon\circ\upphi$ and any map $\yo(C')\to\yo(C)$ being determined by its behavior on sets according to the corresponding property of compactological sets. 
    We will use this fact to show injectivity of the two assignments we construct.

\smallskip
    
    \textcircled{4} Using the canonical epimorphism $\upupsilon\colon F(\ast)\onto hF(\ast)$, and that it is natural in $F$, any map $f\colon hF(\ast)\to C$ determines a unique map of sets $f\circ\upupsilon\colon F(\ast)\to C$. Further, by definition, any $\upphi\in F(C')$ determines a map of compactological sets $C'\to hF(\ast)$, inducing, by composition, an element $C'\to hF(\ast)\to C \in \yo(C)(C')$. It follows that there is a canonical map
    \[
        \upxi\colon \Cmp_\upkappa(hF(\ast), C)\into \PSh{\Cmp_\upkappa}(F, \yo(C)),
    \]
    which is an injection of sets by \textcircled{3} and whose naturality in both components follows by commutativity of pre- and postcomposition, which is in this case uniquely determined by its behavior on underlying sets.

\smallskip

    \textcircled{5} It can be checked from the definitions that $h(\yo(C)(\ast)) \cong C$; employing idempotence of $h\colon \cgTop\onto\CGWH$, the fact that $\yo(C)(\ast)\cong C$ holds as sets and that any compactological space carries the final compactology with respect to all maps to it from compact spaces. We obtain another map of sets
    \[
        \tilde\upxi: \PSh{\Cmp_\upkappa}(F, \yo(C)) \to \Cmp_\upkappa(hF(\ast), h(\yo(C))(\ast)) \cong \Cmp_\upkappa(hF(\ast), C).
    \]
    This map too is injective as the definition of $\upupsilon_C\colon C\to h(C)$ implies that $\tilde\upxi(\upeta)\circ\upupsilon\colon F(\ast)\onto hF(\ast)\to C$ recovers the behavior of $\upeta\colon F\to\yo(C)$ on underlying sets. 

\smallskip
    
    \textcircled{6} The preservation of the underlying behavior by both $\upxi$ and $\tilde\upxi$ implies that 
    \[
        \upxi\circ\tilde\upxi = \id_{\PSh\Cmp(F, \yo(C))}\;\; \text{ and }\;\;\tilde\upxi\circ\upxi = \id_{\Cmp(hF(\ast), C)}
    \]
    hold. In particular, $(\upxi, \tilde\upxi)$ is the desired bijection and the two class-functors are adjoint.
\end{proof}

\begin{lem}\label{lem:cmp-yon-ref-prod}
    The (class-sized) reflector $\PSh{\Cmp_\upkappa}\to\Cmp_\upkappa$, left adjoint to the Yoneda embedding, commutes with finite products.
\end{lem}

\begin{proof}
    Using pointwise computation of limits in functor categories, products in the large category $\PSh{\Cmp_\upkappa}$ exist and the assignment $F\mapsto F(\ast)$ commutes with finite products. We claim that $h(-)$ preserves finite products as well. 
    
    \smallskip

    We derive this from the explicit construction of $hX$ as $X/R$, where $R$ is the smallest equivalence relation on $X$ that is closed as a subset of the compactly generated product $X\times_{\catstyle{CG}} X$. We will just write $\times$ and assume all products to be taken in $\catstyle{CG}$, and hence equivalently in $\CGWH$ whenever all factors are weakly Hausdorff. Let $X, Y$ therefore be two compactly generated spaces and let $R_X, R_Y$ and $R_{X\times Y}$ be the smallest closed equivalence relation on $X$, $Y$ and $X\times Y$ respectively. For every $y \in Y$ there exists an inclusion $\upiota_y: X\times X \into X\times Y \times X \times Y$, mapping $(x_1, x_2)$ to $(x_1, y, x_2, y)$. By continuity of this inclusion, $\upiota_y^{-1}(R_{X\times Y})$ is closed in $X\times X$ and can be seen to form an equivalence relation on $X$. Thus it contains $R_X$. The same argument applies to the corresponding inclusions of $Y\times Y$ and thus, by transitivity, if $x_1\sim_X x_2$ and $y_1\sim_Y y_2$, then $(x_1, y_1) \sim_{X\times Y} (x_2, y_1) \sim_{X\times Y} (x_2, y_2)$. We derive that $R_X\times R_Y$ embeds naturally into $R_{X\times Y}$, and as the product of two closed sets it will be closed in $X\times X\times Y\times Y\cong X\times Y\times X\times Y$. Because $R_{X\times Y}$ is by definition the smallest such relation, we obtain that $R_X\times R_Y \cong R_{X\times Y}$ and hence $h(X\times Y)\cong hX\times hY$, using that quotients commute with finite products in $\catstyle{CG}$.
    
    \smallskip

    Finally, the compactology of $hF(\ast)$ is given by sets of the form $\bigcup_{i=1}^n\ran\upphi_i$ for $\upphi_i\colon \yo(K_i)\to F$ and $K\in\CHaus_\upkappa$. Again by pointwise computation of limits it follows that for presheaves $F$ and $G$, the compactology of the product $h(F\times G)(\ast)$ contains the relevant cylinder sets and is bounded by these and hence agrees with the compactology of the product $hF(\ast)\times hG(\ast)$.
\end{proof}

\begin{prop}\label{prop:cmp-cart-clos}
    The category of ($\upkappa$-)compactological sets is cartesian closed and the internal hom of $\CGWH$ is preserved under the inclusion $\CGWH\into\Cmp$.
\end{prop}

\begin{proof}
    By Wood \cite[Thm 9]{Wood} a total category is cartesian closed if and only if the left adjoint to the Yoneda embedding preserves finite products. The direction that we need is the easier one to show and hence we will repeat Wood's argument here. We aim to apply the adjoint functor theorem for total categories, see \cite[p.~372]{StreetWaltersYoneda}, to show that the product functor $C \times -$ admits a right adjoint for any $C\in\Cmp_\upkappa$. For this we use that we may compute colimits of compactological sets in presheaves and reflect them down using the left adjoint $L\colon \PSh\Cmp \to \Cmp$. Now, using that limits and colimits in functor categories are computed pointwise and the binary product in $\Set$ is cocontinuous, we obtain that $\yo(C)\times \colim_i \yo(C_i) \cong \colim_i \yo(C)\times \yo(C_i)$ for any diagram $i\mapsto C_i$. Using that $L$ preserves finite products according to \cref{lem:cmp-yon-ref-prod} we get
    \begin{align*}
        C\times \colim_i C_i &\cong L(\yo(C))\times L(\colim_i \yo(C_i))
        \cong L(\yo(C)\times \colim_i \yo(C_i)) \\
        &\cong L(\colim_i \yo(C)\times \yo(C_i))
        \cong \colim_i L(\yo(C))\times L(\yo(C_i)) \cong \colim_i C\times C_i.
    \end{align*}
    and accordingly $C\times -$ is cocontinuous. As $L$ preserves potentially large colimits, $C\times -$ preserves all existing potentially large colimits as well. The aforementioned version of the adjoint functor theorem (\cite[p.~372]{StreetWaltersYoneda} or \cite[Thm 1]{Wood}) now says that any such functor out of a total category preserving all (existing) large colimits admits a right adjoint and we may derive cartesian closedness of $\Cmp_\upkappa$.

    \smallskip

    To see that $\Cmp$ admits an internal hom as well, observe first that $\Cmp$ is obtained as the (class-sized) union of all $\Cmp_\upkappa$ and all limits and colimits may be computed within $\Cmp_\upkappa$ for $\upkappa$ large enough. Now for $\upkappa<\upkappa'$ and the coreflective adjunction $\upiota\colon \Cmp_\upkappa \adjarrow \Cmp_{\upkappa'} :\!\uppi$, we have
    \begin{align*}
        \Cmp_{\upkappa'}(\upiota(C), \upiota([C', C''])) &\cong \Cmp_\upkappa(C, [C', C''])
        \cong \Cmp_\upkappa(C\times C', C'') \\
        &\cong \Cmp_{\upkappa'}(\upiota(C)\times\upiota(C'), \upiota(C''))
        \cong \Cmp_{\upkappa'}(\upiota(C), [\upiota(C'), \upiota(C'')]),
    \end{align*}
    using full faithfulness and commutativity with finite products of $\upiota$. 

    \smallskip

    We may deduce that $\upiota([C', C'']) \cong [\upiota(C'), \upiota(C'')]$ holds whenever $\upkappa$ is greater than the cardinality of $[\upiota(C'), \upiota(C'')]$, as in this case $[\upiota(C'), \upiota(C'')]$ is generated by $\upkappa$-small compact spaces. In particular, for any $C\in\Cmp$ one can construct compatible internal hom functors $[C, -]\colon \Cmp_\upkappa\to\Cmp$ for all $\upkappa$, which yield an internal hom $[C, -]\colon \Cmp\to\Cmp$ by the universal property of $\Cmp\cong\colim_\upkappa\Cmp_\upkappa$. That this functor is right adjoint to the product functor $-\times C$ then follows by computation in $\Cmp_\upkappa$ for $\upkappa$ large enough.
    
    \smallskip

    That the internal hom of $\CGWH$ is preserved by the inclusion $\CGWH\into\Cmp$ finally follows by the following chain of isomorphisms:
    \begin{align*}
        \Cmp(\upiota(K), \upiota([X, Y])) &\cong \CGWH(K, [X, Y])
            \cong \CGWH((\upiota(K\times X))_{\operatorname{top}}, Y) \\
            &\cong \Cmp(\upiota(K\times X), \upiota(Y)) \\
            &\cong \Cmp(\upiota(K)\times \upiota(X), \upiota(Y))
            \cong \Cmp(\upiota(K), [\upiota(X), \upiota(Y)]),
    \end{align*}
    for $K\in\CHaus$ and $X, Y\in\CGWH$. 
    Besides the adjunction of \cref{CGWH-COMPL-ADJ}, this uses the continuity of $\upiota$ by virtue of being a right adjoint and the canonical adjunction defining cartesian closure.

    \smallskip

    By density of $\CHaus$ in $\Cmp$, this chain suffices to show that $[\upiota(X), \upiota(Y)]\cong\upiota([X, Y])$ holds.
\end{proof}

\begin{lem}\label{lem:yo-internal-hom}
    The embedding $\yo: \Cmp\into\Cond$ preserves the internal hom.
    In particular, the internal hom between quasiseparated condensed sets is quasiseparated.
\end{lem}

\begin{proof}
    This follows by the same yoga of adjunctions:
    \begin{align*}
        \Cond(\yo(C), \yo([C', C''])) &\cong \Cmp(C, [C', C''])
        \cong \Cmp(C\times C', C'') \\
        &\cong \Cond(\yo(C)\times\yo(C'), \yo(C''))\\
        &\cong \Cond(\yo(C), [\yo(C'), \yo(C'')]).
    \end{align*}
    Again, density of $\CHaus$ in both categories implies the claim.
\end{proof}

We continue with a sequence of properties of $\Cmp$ and $\Cmp_\upkappa$ illustrating that both are `well-behaved'. As the results do not depend on $\upkappa$ we leave it implicit.

\begin{prop}\label{prop:disj}
    The category $\Cmp$ is (infinitary) disjunctive. Explicitly, in addition to admitting coproducts and all finite limits, coproducts in $\Cmp$ are disjoint and stable under pullback. More explicitly still, for all indexing sets $I$ and elements $j, k, \ell\in I$ with $k\neq\ell$ the canonical inclusions $\upiota_j: C_j\into \coprod_{i\in I} C_i$ are monic, it holds that
    \[
        C_k\times_{\coprod_{i\in I} C_i} C_\ell \cong \varnothing,
    \]
    and
    \[
        \left(\coprod_{i\in I} C_i\right)\times_{C'} C \cong \coprod_{i\in I}\left(C_i\times_{C'} C\right).
    \]
\end{prop}

\begin{proof}
    Monicity of the canonical inclusions into the coproduct is evident from \cref{prop:cmp-prod}. In fact, the inclusions are regular.
    Using that the forgetful functor $\Cmp\to\Set$ is continuous and preserves coproducts, the underlying set of $C_k\times_{\coprod_{i\in I} C_i} C_l$ is the corresponding pullback of underlying sets, which is empty, as coproducts in $\Set$ are disjoint.

    \smallskip

    Finally, stability under pullback follows by observing that products distribute over coproducts in $\Cmp$ and that the pullback is the regular subobject of this product, which equalizes the two relevant morphisms.
\end{proof}

Below, we recall the definition of regularity of categories in the sense of \cite[Dfn A.5.1]{BorcBourn04} and \cite[p.~134]{BarrGrilletOsdol}, which can be understood as a property providing a well-behaved notion of quotient maps and image factorizations of maps. It is one of the properties making topoi particularly well-behaved and it is closely related to the concept of a \emph{left quasi-abelian category} which is essentially an additive (and) regular category. The only difference is that in a left quasi-abelian category every morphism is required to have a cokernel, see \cite[p.~458--459]{HKRW} for more details. The concept dualizes to coregularity and a regular and coregular additive category is precisely a quasi-abelian category. We point out that quasi-abelian categories appear naturally in the classical homological theory of locally convex or bornological vector
spaces \cite{Schneiders99, ProsmansSchneiders, Prosmans}.

\begin{dfn}\label{def:regularity}
    A morphism in a category is called a \emph{regular epimorphism} if it is the coequalizer of some parallel pair.
    A category $\catC$ is called \emph{regular} if
    \begin{enumerate}
        \item it admits finite limits,\vspace{2pt}
        \item for any $f\colon C\to D$ in $\catC$ the kernel pair
            \[\begin{tikzcd}
                C\times_D C
                    \ar[r, dashed]
                    \ar[d, dashed]
                    \ar[phantom, dr, "\lrcorner" very near start] &
                C
                    \ar[d, "f"] \\
                C
                    \ar[r, "f"'] &
                D
            \end{tikzcd}\]
            admits a coequalizer $C \onto C/C\times_D C$,\vspace{2pt}
        \item regular epimorphisms are stable under pullback.
    \end{enumerate}
    A functor $F\colon \catC\to\catD$ between regular categories is called \emph{regular} if it preserves finite limits and regular epimorphisms. \hfill\diam
\end{dfn}

It is equivalent to define regularity as the existence of finite limits and a pullback-stable factorization system consisting of regular epimorphisms and all monomorphisms. An important example for the above definition is given by the following result, which we cite without proof.

\begin{thm}[\protect{\cite[Thm 3.1]{CGWH}}]\label{thm:cgwh-bireg}
    The category $\CGWH$ is regular and coregular. \qed
\end{thm}

We will transfer this statement to $\Cmp$ using the faithful functor $(-)_{\operatorname{top}}\colon \Cmp\onto\CGWH$, which exhibits $\Cmp$ as a concrete category of topological spaces. As $\Cmp$ admits all limits and colimits, regularity corresponds to showing pullback stability of regular epimorphisms.

\begin{lem}\label{lem:repi-pb}
    The category $\Cmp$ of compactological sets is regular.
\end{lem}

\begin{proof}
    By cocontinuity and finite continuity of the forgetful functor $\topuly -\colon \Cmp\to\CGWH$, both the pullback of a coequalizer and the coequalizer of a pullback are preserved and by regularity of $\CGWH$ and the general fact that the pullback of a kernel pair is the kernel pair of the pullback, it follows that the pullback of a regular epimorphism of compactological sets is a regular epimorphism in $\CGWH$. Accordingly, it only remains to show that it is a regular epimorphism with regard to the compactology. For this, we fix the following pullback square with a regular epimorphism $f\colon X\onto Y$:
    \[\begin{tikzcd}
        Z\times_Y X
            \ar[r, dashed, "g'"]
            \ar[d, dashed, "f'"']
            \ar[phantom, dr, "\lrcorner" very near start] &
        X
            \ar[d, two heads, "f"] \\
        Z
            \ar[r, "g"] &
        Y.
    \end{tikzcd}\]
    We need to show that $f'(\Comp_{Z\times_Y X}) = \Comp_Z$. For this, fix an arbitrary $K_Z\in \Comp_Z$.
    By definition, $g(K_Z)\in\Comp_Y$ and there exists some $K_X\in\Comp_X$ with $f(K_X) = g(K_Z)$ by regularity of $f$. Using functoriality and continuity of the pullback, we obtain an embedding $K_Z\times_{g(K_Z)} K_X \into Z\times_Y X$ and using that $\upiota\colon \CHaus\into\Cmp$ preserves limits, $K_Z\times_{g(K_Z)} K_X$ lies in the image of $\upiota$. In particular, it must be an element of the compactology of $Z\times_Y X$ and by the fact that $K_Z\times_{g(K_Z)} K_X \onto K_Z$ is a regular epimorphism of cgwh-spaces, we obtain the desired equality $f'(K_Z\times_{g(K_Z)} K_X) = K_Z$.
\end{proof}

% Fleck der Schande

\begin{lem}\label{lem:rmono-po}
    The category $\Cmp$ of compactological sets is coregular.
\end{lem}

\begin{proof}
    Again it suffices to show that regular monomorphisms are stable under pushout.
    The proof of this is similar to the pullback stability of regular epimorphisms. We fix the following pushout for a given regular monomorphism $f\colon X\into Y \in\Cmp$.
    \[\begin{tikzcd}
        X
            \ar[r, "g"]
            \ar[d, "f"', hook]
            \ar[phantom, dr, "\ulcorner" very near end] &
        Z
            \ar[d, dashed, "f'"] \\
        Y
            \ar[r, dashed, "g'"'] &
        Y \sqcup_X Z
    \end{tikzcd}\]
    By the same argument as above, we know that $\topuly{f'}$ is a regular monomorphism in $\CGWH$ and it remains to show that every $f^{-1}(K_{Y\sqcup_X Y}) \subseteq Z$ for $K_{Y\sqcup_X Z}\in \Comp_{Y\sqcup_X Z}$ is in $\Comp_Z$. From \cref{cor:cmp-po} we know that $\Comp_{Y\sqcup_X Z}$ is generated by $g'(\Comp_Y) \cup f'(\Comp_Z)$. Because $f'$ is an embedding of $\CGWH$ spaces, $(f')^{-1}\circ f' = \id_{Z}$, and we only need to show that $(f')^{-1}(g'(K_Y))\in\Comp_Z$ holds for all $K_Y\in\Comp_Y$. We choose $K_Y \in\Comp_Y, K_X=f^{-1}(K_Y) \in \Comp_X, K_Z = g(K_X) \in \Comp_Z$ and using the fact that $f$ and hence $f'$ restrict to homeomorphisms onto their image, we have $(f')^{-1}(g'(K_Y)) = (f')^{-1}(g'(K_Y) \cap \ran f') = g(f^{-1}(K_Y \cap \ran f)) = g(f^{-1}(K_Y)) = K_Z$.
\end{proof}

\smallskip

We conclude the list of `exactness' properties we have proven thus far (regularity, disjunctivity, etc.) with the following property of the same kind:

\begin{prop}\label{prop:cmp-coherent}
    The category $\Cmp$ of compactological sets is a well-powered infinitary coherent category. Explicitly it is a regular category such that for every compactological set $C\in\Cmp$ the poset $\operatorname{Sub}(C)$ of isomorphism classes of monomorphisms $C'\into C$
    \begin{enumerate}
        \item is small
        \item and admits arbitrary unions,
        \item which are stable under pullbacks.
    \end{enumerate}
\end{prop}

To prove \cref{prop:cmp-coherent}, we introduce the following lemma.

\begin{lem}\label{lem:cmp-subobject}
    For a compactological set $C = (X, \Comp)$ the poset $\operatorname{Sub}(C)$ is isomorphic to the poset of those subdiagrams $\Comp' \subseteq \Comp$, which are downward closed and contain all finite unions of their elements.
\end{lem}

\begin{proof}
    It follows directly from \cref{DFN-COMP-4} that any such subdiagram $\Comp'$ is a valid compactology on the set $\bigcup\Comp' =: X'$ that it covers. That the inclusion map $(X', \Comp') \into (X, \Comp)$ is continuous follows immediately as a set $M\subseteq X$ is closed iff $\Comp\cap M\subseteq \Comp$, which implies that $X'\cap M \cap \Comp' = M\cap\Comp' \subseteq \Comp'$ by the fact that $\Comp'$ is downward closed in $\Comp$. That it preserves the compactology is by definition. This constructs a canonical map from the poset of subdiagrams to the poset of compactological subobjects, which can be seen to be monotone by application of the same argument to $\Comp'\subseteq \Comp'' \subseteq \Comp$.
    
    \smallskip

    %We demonstrate that any monomorphism $f\colon (X', \Comp') \into (X, \Comp)$ is determined up to isomorphism by the subdiagram $f(\Comp') \subseteq \Comp$ and that this subdiagram is of the desired form, which implies that this canonical map is an bijection. We also show that the existence of a monomorphism $(X', \Comp') \into (X'', \Comp'')$ making the evident diagram of inclusions commute, implies $f(\Comp')\subseteq g(\Comp'')$ for the respective monomorphisms $f$ and $g$.

    %\smallskip
    
    To show that this map is an isomorphism of posets, we may assume without loss of generality that for an arbitrary monomorphism $f\colon (X', \Comp')\into (X, \Comp)$ the set $X'$ is a subset of $X$, because compactologies may be isomorphically pushed and pulled along bijections and hence $\Comp'$ may be pushed along $f$ to $\ran f\subseteq X$, such that the inclusion remains a morphism of compactological sets.
    With this it follows that $\Comp'$ is a subdiagram of $\Comp$ as the inclusion map $f$ must fulfil $\Comp' = f(\Comp') \subseteq \Comp$.

    \smallskip
    
    In particular, for a chain $(X', \Comp') \into (X'', \Comp'') \into (X, \Comp)$ of monomorphisms on subset inclusions we have $\Comp'\subseteq \Comp''$.
    Now, using the fact that continuity of maps between compactological spaces may be checked on the elements of the compactology, the requirement that $f$ is continuous amounts to requiring
    $$
   \forall\:K\in\Comp'\colon  f|_K^{-1}(\Comp \cap f(K)) = \Comp\cap K \subseteq \Comp' \cap K.
    $$
    Hence, we obtain $\Comp\cap K \subseteq \Comp'\cap K \subseteq \Comp\cap K$ from the fact that $K$ lies in both compactologies and compactologies are stable under intersection. Rephrasing this we get that $\Comp'\subseteq\Comp$ must be downward closed.
    Further, because $\Comp'$ covers $X'$, $X' = \bigcup \Comp'$ uniquely determines the set $X'$ and because $\Comp'$ is a compactology it is closed under finite unions. Accordingly, any monomorphism $f\colon (X', \Comp') \into (X, \Comp)$ as above is described by the downwards and finite union closed subdiagram $\Comp'\subseteq\Comp$.
\end{proof}

\begin{proof}[Proof of \protect{\cref{prop:cmp-coherent}}]\leavevmode \textcircled{1} That $\Cmp$ is well-powered, i.e., that the subobject posets of $\Cmp$ are small, follows directly from \cref{lem:cmp-subobject}, which implies the existence of an injection $\operatorname{Sub}_{\Cmp}(C)\into \operatorname{Sub}_\Set(\mathcal{P}(U(C)))$ into subsets of the powerset of $U(C)$.

\smallskip

\textcircled{2} From \cref{lem:cmp-subobject} it is evident that any collection of subobjects, given by subdiagrams $(\Comp_i\subseteq \Comp)_{i\in I}$, admits a union of diagrams $\mathscr U := \bigcup_{i\in I} \Comp_i$, which will be again downward closed in $\Comp$. Closing $\mathscr U$ under finite unions, it determines a unique subobject $(X_{\mathscr U}, \overline{\mathscr U})$ of $C = (X, \Comp)$ and because the subdiagram $\Comp'$ corresponding to any other subobject of $C$ is closed under finite union, $\mathscr U$ embeds into $\Comp'$ if and only if its closure under finite union $\overline{\mathscr U}$ does.

\smallskip

\textcircled{3} Finally, we have to show that these unions are stable under pullback. For this, we describe the pullback of a monomorphism in terms of the above characterization of subobjects as subdiagrams. From continuity of $U\colon \Cmp\to\Set$ we may conclude that on sets the pullback of $(X'\subseteq X, \Comp')\into(X, \Comp)$ along $f: (Y, \Comp_Y)\to (X, \Comp)$ includes as $f^{-1}(X')$ into $Y$. Thus the pullback necessarily corresponds to some downwards and finite union closed subdiagram of $\Comp_Y\cap f^{-1}(X')$. By the universal property, this subdiagram must map to a subdiagram of $\Comp'$ and it must be maximal in this respect. Thus the pullback $f^*(\Comp')$ of the subobject $\Comp' \subseteq \Comp$ is given by $\Comp_Y \cap f^{-1}(\Comp')$. 

\smallskip
    
Now, replacing $(X', \Comp')$ by a family of subdiagrams $(X_i, \Comp_i)$, we have $f^{-1}(\bigcup_{i\in I} \Comp_i) = \bigcup_{i\in I} f^{-1}(\Comp_i)$ and 
\begin{align*}
\Comp_Y\cap \bigcup_{i\in I} f^{-1}(\Comp_i) &= \{K_Y\cap f^{-1}(K_X) \mid K_Y\in\Comp_Y, K_X\in\bigcup_{i\in I} \Comp_i \} \\
&= \bigcup_{i\in I} \{K_Y\cap f^{-1}(K_X) \mid K_Y\in\Comp_Y, K_X\in \Comp_i \} = \bigcup_{i\in I} \Comp_Y \cap \Comp_i.
\end{align*}
From the fact that unions commute with inverse images and that intersections distribute over unions it follows that $\Comp_Y\cap f^{-1}(\overline{\bigcup_{i\in I} \Comp_i}) = \overline{\bigcup_{i\in I} \Comp_Y \cap f^{-1}(\Comp_i)}$, where $\overline{(-)}$ denotes the closure under finite unions. From the description of the subobject-union $\bigvee_{i\in I} \Comp_i$ of the subobjects $(X_i, \Comp_i)$ in terms of their diagrams it follows that $f^*(\bigvee_{i\in I} (X_i, \Comp_i)) = \bigvee_{i\in I} f^*((X_i, \Comp_i))$.
\end{proof}

\smallskip

Aggregating the results of this section, the next logical question is how close the category of $\upkappa$-compactological sets is to forming a topos. It is obvious that $\Cmp_\upkappa$ admits a generator in the guise of the one point space. It is well-powered and coherent by the above proposition.
Finally it is infinitary extensive, as this is a weaker version of the disjunctivity proven in \cref{prop:disj}. This implies according to Giraud's axioms that $\Cmp_\upkappa$ is a (Grothendieck) topos if and only if it is Barr-exact, i.e., if and only if every internal equivalence relation is a kernel pair of some morphism. Unsurprisingly, we obtain that this cannot be the case.

\smallskip

To see this, assume that $1\into\Omega$ is a subobject classifier, which would necessarily exist if $\Cmp_\upkappa$ was a topos. Then any subobject $C\times_\Omega 1$ must be regular and closed, being the pullback of a regular monomorphism with closed range. %image. 
However, the open sets of any cgwh-space are cgwh and the functor $\upiota$ includes them into $\Cmp$ to form a counterexample. This fact is closely related to the coequalizer of an equivalence relation on a compactological set corresponding to the quotient by smallest closed equivalence relation containing it. 

\smallskip

As an example of this failure, let $\mathbb{Q}\into\mathbb{R}$ be the canonical inclusion of the two spaces treated as compactological sets. This gives rise to an equivalence relation on $\mathbb R$ generated by $\mathbb Q\times \mathbb Q$, the categorical quotient of which is necessarily the unique compactological set on a single point. This is an example of a non-effective equivalence relation in $\Cmp$, directly violating Barr-exactness.

\smallskip

However, one may complete $\Cmp_\upkappa$ to an exact category by passing to the appropriate category of congruences in $\Cmp_\upkappa$. For a detailed review of, among others, regular categories and their completion to exact categories we refer to \cite{Menni2000ExactCompletions}. We will see that this completion is in fact the (Grothendieck) topos of $\upkappa$-condensed sets. Because compactological sets are not locally cartesian closed\,---\, the quotient of $\{0,1\}\times \Real$ by $(0, \frac p q) \times (1, \frac p q)$ over $\Real$ is $\Real$, its pullback to $\Real\setminus\QQ$ is $\{0, 1\}\times (\Real\setminus\QQ)$, preventing pullback along $\Real\setminus\QQ\into\Real$ from having a right adjoint \,---\,this provides a naturally occurring example that the assumptions of \cite[Thm 11.3.3]{Menni2000ExactCompletions} are only sufficient and not necessary. As a word of caution we note here that the ex/reg completion of $\Cmp$ is \emph{not} a topos as it lacks a subobject classifier.

\begin{dfn}[\protect{\cite[Dfn 2.5.1, Sec 3.4]{Menni2000ExactCompletions}}]\label{def:ex/reg} Let $\catC$ be a regular category according to \cref{def:regularity}.
\begin{enumerate}
\item[(1)] $\catC$ is called \emph{Barr-exact} if every internal equivalence relation is a kernel pair.

\vspace{2pt}

\item[(2)]  The \emph{ex/reg completion} of $\catC$ is a Barr-exact category $\catC_{\operatorname{ex/reg}}$ together with a fully faithful and regular functor $\catC\into \catC_{\operatorname{ex/reg}}$, such that for any other Barr-exact category $\catD$ the precomposition with the inclusion functor induces an equivalence between the categories $\Fun_{\operatorname{reg}}(\catC, \catD)$ and $\Fun_{\operatorname{reg}}(\catC_{\operatorname{ex/reg}}, \catD)$ of regular functors.  \hfill\diam
\end{enumerate}    
    %Let $\catC$ be a regular category according to \cref{def:regularity}. It is called \emph{(Barr-) exact} if every internal equivalence relation is a kernel pair.

    %The ex/reg completion of $\catC$ is a (Barr-)exact category $\catC_{\operatorname{ex/reg}}$ together with a fully faithful and regular functor $\catC\into \catC_{\operatorname{ex/reg}}$, such that for any other (Barr-)exact category $\catD$ the precomposition with the inclusion functor induces an equivalence between the categories $\Fun_{\operatorname{reg}}(\catC, \catD)$ and $\Fun_{\operatorname{reg}}(\catC_{\operatorname{ex/reg}}, \catD)$ of regular functors.  \hfill\diam
\end{dfn}

\smallskip

The ex/reg completion is necessarily unique up to a canonical equivalence of categories.
By the fact that exact categories are a full subcategory of the 2\nobreakdash-category of regular categories, it follows that any exact category is its own ex/reg completion.
It has been shown by Succi Cruciani \cite{Cruciani1975RegEx} that the ex/reg completion of $\catC$ always exists and can be constructed as a category of congruences in $\catC$ with function-like relations between objects. It follows from this construction that whenever a subcategory $\catC\subseteq\catD$ of a regular category $\catD$ is closed under finite limits and subobjects and is regular itself, $\catC_{\operatorname{ex/reg}}$ will be contained as a full subcategory in $\catD_{\operatorname{ex/reg}}$, as any congruence in $\catC$ will be a congruence in $\catD$ and any relation in $\catD$ between two congruences from $\catC$ lifts to a relation in $\catC$. Finally, pullbacks in $\catC$ can be computed in $\catD$ and thus the two categories have the same notion of composition of relations. 

\smallskip

In this light we return to \cref{thm:cmp-condqs} to obtain a characterization of condensed sets as formal quotients of compactological sets, which, in the $\upkappa$-restricted case, also captures the concrete failure of $\Cmp_\upkappa$ to be a (Grothendieck) topos. As remarked above, $\Cmp_{\operatorname{ex/reg}}$ does not form a topos.

\thmcmpregex

%\begin{thm}\label{thm:cmp-reg-ex}
 %   The category $\Cond_\upkappa$ of $\upkappa$-condensed sets is the \emph{ex/reg}-completion of the regular category $\Cmp_\upkappa$ of $\upkappa$-compactological sets to a Barr-exact category.
  %  It follows that $\upkappa$-condensed sets are, up to equivalence, compactological equivalence relations with those compactological relations between them that are function-like and constant on equivalence classes.
   % \missing: Reference
%\end{thm}

\begin{proof}\textcircled{1}
    By the density theorem \cite[Thm 6.5.7]{CatContext} any (pre)sheaf is a canonical colimit of representables. Using the usual formula for colimits it follows that any condensed set is a quotient of a coproduct of representables and hence, by \cref{cor:ind-condqs}, a quotient of a compactological set. Using further the regularity of sheaf-categories, the coequalizer describing this quotient may be expressed as the coequalizer of its kernel pair and accordingly as a quotient by a congruence. This congruence is quasiseparated as %by the simple fact that 
    it is a subobject of a product of quasiseparated objects.

    \smallskip

    By the fact that the ex/reg completion is monotone with respect to `nice' full subcategories, as remarked previously, $(\Cmp_\upkappa)_{\operatorname{ex/reg}}$ canonically embeds into $(\Cond_\upkappa)_{\operatorname{ex/reg}}$, because quasiseparated condensed sets are closed under all limits and subobjects.
    Furthermore, the ex/reg completion is idempotent and hence Barr-exactness of the Grothendieck topos $\Cond_\upkappa$ implies that $(\Cond_\upkappa)_{\operatorname{ex/reg}} \cong \Cond_\upkappa$. We may derive that the ex/reg completion of $\Cmp_\upkappa$ is equivalent to a full subcategory of $\Cond_\upkappa$. However, by the initial remark, $\Cond_\upkappa$ is spanned by quotients of such congruences, which means that the functor $(\Cmp_\upkappa)_{\operatorname{ex/reg}}\to\Cond_\upkappa$ must also be essentially surjective. 

    \smallskip
    
    %That the functors $\Cmp_\upkappa\into(\Cmp_\upkappa)_{\operatorname{ex/reg}}$ and $\Cmp_\upkappa\into\Cond_\kappa$ agree up to natural isomorphism with respect to this equivalence follows directly from the fact that they must agree on the objects $C\in\Cmp_\upkappa$, embedded as $\Updelta_C\into C\times C$ and $\yo(C)$ respectively, because on functor of the equivalence sends $\Updelta_C\into C\times C$ to $\yo(C)/\yo(\Updelta_C) \cong \yo(C)$. Any other object is identified as a coequalizer of such objects and the equivalence relation trivially preserves these.

    %\smallskip

%\begin{lem}\label{lem:cmp-reg-incl}\NOTE{Sven}{Include Lem 4.10 into Thm 4.9.}
 %   The inclusion $\yo: \Cmp\into\Cond$, embedding compactological sets into their \emph{ex/reg}-completion, preserves coequalizers of kernel pairs and is hence a regular functor.
%\end{lem}
    \textcircled{2} The regularity of the embedding $\yo\colon \Cmp_\upkappa\into\Cond_\upkappa$ follows abstractly. We provide the following explicit proof for the sake of completeness. 
    
    \smallskip
    
    Let $f\colon C\to D$ be a regular epimorphism of compactological sets. Then $f$ is the coequalizer of its kernel pair $C\times_D C$.
    As $\yo$ is a right adjoint it preserves limits and hence the kernel pair of $\yo(f)$ is given by $\yo(C\times_D C)$.
    It is a fact that in any regular category, any morphism $\upphi$ factors as the coequalizer of its kernel pair, followed by a monomorphism representing the coimage $\operatorname{coran}(f)$. In particular, this is true for $\yo(f)$. However, $\cmpuly-$ preserves colimits and hence 
    \[
        \cmpuly{\operatorname{coran}(\yo(f))} = \operatorname{coran}(\cmpuly{\yo(f)}) = \operatorname{coran}(f) = f,
    \]
    as $f$ is a regular epimorphism. Thus, $\yo(f)$ is the quasiseperation of the coequalizer and by the coimage factorization, the coequalizer is a subobject of $\yo(f)$. However, subobjects of quasiseparated condensed sets are quasiseparated and hence their own quasiseperation. It follows that the coequalizer of the kernel pair of $\yo(f)$ is $f$ itself and the claim follows.

    \smallskip

    \textcircled{3} That this result extends to the directed colimits over all cardinals follows from the fact that any object $X\in\Cond$ arises as $\upiota_\upkappa(X)$ for some $\upiota_\upkappa\colon \Cond_\upkappa\into\Cond$ and that $\upkappa$ may be chosen large enough as that it admits an epimorphism from some $\upkappa$\nobreakdash-compactological set $C$ for which the inclusion $\yo(C)$ of $C$ into $\Cond$ factors over $\Cond_\upkappa$. In this case $X$ arises as the coequalizer of a compactological congruence $C'\rightrightarrows C$, both in $\Cond_\upkappa$ and in $\Cond$. In the same manner as in \cref{thm:cmp-condqs} we construct a cone with tip $\Cond$ under the directed diagram of all ex/reg completions $(\Cmp_\upkappa)_{\operatorname{ex/reg}}$, such that the arrow out of the colimit is fully faithful. By the above argument this functor will then also be essentially surjective.
\end{proof}

\medskip

%\printbibliography[heading=bibintoc]
 %\bibliographystyle{amsplain}
 %\bibliography{Ref-cmp}

\begin{center}
{\sc Use of Machine Learning Tools}
\end{center}

\small

During the preparation of this paper the authors used DeepL and ChatGPT in order to improve readability and language.  Additionally, the completed article was proofread using the ChatGPT Astra model. Based on its findings minor corrections were made, but no significant content or mathematical arguments have been changed. The authors reviewed the whole article before submission and take full responsibility for the content of the publication.

\normalsize

\smallskip

\begin{center}
{\sc Acknowledgments}
\end{center}

\small

The authors would like to thank Julian Holstein, Julius Mann and Thomas Stempfhuber for valuable input and feedback during the whole process of writing this article.

\end{document}